\documentclass[11pt,reqno]{amsart}

\usepackage[T1]{fontenc}
\usepackage[utf8]{inputenc}
\usepackage{lmodern,microtype}
\usepackage[a4paper,margin=31mm]{geometry}
\usepackage{amsmath,amssymb,amsthm,mathtools}
\usepackage{enumitem}
\usepackage{aliascnt}
\usepackage{xcolor,booktabs,tabularx}

\usepackage[hidelinks]{hyperref}
\usepackage[nameinlink,noabbrev]{cleveref}

\numberwithin{equation}{section}

\newtheorem{theorem}{Theorem}[section]
\newaliascnt{proposition}{theorem}
\newtheorem{proposition}[proposition]{Proposition}
\aliascntresetthe{proposition}
\newaliascnt{lemma}{theorem}
\newtheorem{lemma}[lemma]{Lemma}
\aliascntresetthe{lemma}
\newaliascnt{corollary}{theorem}
\newtheorem{corollary}[corollary]{Corollary}
\aliascntresetthe{corollary}
\newaliascnt{question}{theorem}
\newtheorem{question}[question]{Question}
\aliascntresetthe{question}

\theoremstyle{definition}
\newaliascnt{definition}{theorem}

\aliascntresetthe{definition}

\theoremstyle{remark}
\newaliascnt{remark}{theorem}
\newtheorem{remark}[remark]{Remark}
\aliascntresetthe{remark}

\crefname{theorem}{Theorem}{Theorems}
\Crefname{theorem}{Theorem}{Theorems}
\crefname{proposition}{Proposition}{Propositions}
\Crefname{proposition}{Proposition}{Propositions}
\crefname{lemma}{Lemma}{Lemmas}
\Crefname{lemma}{Lemma}{Lemmas}
\crefname{corollary}{Corollary}{Corollaries}
\Crefname{corollary}{Corollary}{Corollaries}
\crefname{question}{Question}{Questions}
\Crefname{question}{Question}{Questions}
\crefname{definition}{Definition}{Definitions}
\Crefname{definition}{Definition}{Definitions}
\crefname{remark}{Remark}{Remarks}
\Crefname{remark}{Remark}{Remarks}

\newcommand{\C}{\mathbb C}
\newcommand{\B}{\mathbb B}
\newcommand{\norm}[1]{\left\lVert #1\right\rVert}
\newcommand{\abs}[1]{\left\lvert #1\right\rvert}
\newcommand{\Ucal}{\mathcal U}

\newcommand{\PSH}{\operatorname{PSH}}

\title{Complex structures of non-autonomous basins}
\author{Luka Boc Thaler}
\thanks{The author was supported by the research program P1-0291 and project J1-70033, all from ARRS, Republic of Slovenia}

\address{L. Boc Thaler: Faculty of Mathematics and Physics, University of Ljubljana, SI--1000 Ljubljana, Slovenia. Institute of Mathematics, Physics and Mechanics, Jadranska 19, 1000 Ljubljana, Slovenia.} \email{luka.boc@fmf.uni-lj.si}


\begin{document}

\begin{abstract}
{We study non-autonomous basins of $F_j$ automorphisms of $\C^2$ satisfying $$a_j\|p\|\leq\|F_j(p)\|\leq b_j\|p\| \qquad (\forall p\in \B)$$ where $0<a_j<b_j<b<1$. We show that if  $b_j^m<a_j$ for some integer $m$ and all $j>0$ then every bounded plurisubharmonic function on the non-autonomous basin is constant. Moreover, if in addition we have $\prod_{i=1}^{j-1}b_i\leq b_j^m$ then the basin is biholomorphic to $\C^2$. 

For the quadratic H\'enon family
$$F_n(z,w)=\bigl(z^2+c^{q_n}w,c^{q_n}z\bigr),
\qquad 0<c<1,$$
we study how the asymptotic behavior of the sequence $q_n\to\infty$ influences the complex structure of the basin. In particular, we identify broad regimes in which the non-autonomous basin is either biholomorphic to $\C^2$ or is a Short $\C^2$, thereby isolating possible sequences for which other complex structures may  occur.
 }
\end{abstract}

\maketitle
\enlargethispage{4pt}

\section{Introduction}
Let $F$ be an automorphism of $\C^2$ which fixes the origin. If the origin is an attracting fixed point, that is, if both eigenvalues of $DF(0)$ have modulus less than one, then the classical work of Rosay--Rudin \cite{RR} asserts that its basin of attraction 
$$\Omega_F=\{p\in \C^2\mid F^n(p)\xrightarrow{n\rightarrow\infty} 0\}$$
is biholomorphic to $\C^2$. This provides a useful method for constructing proper subsets of $\C^2$ that are biholomorphic to $\C^2$, a phenomenon that cannot occur in dimension one. Such a proper subset is called a Fatou-Bieberbach domain. 

Instead of considering autonomous orbits, it is natural to consider the non-autonomous analogue. Given a sequence $(F_n)\subset \operatorname{Aut}(\C^2)$, we write
$$F(n):=F_n\circ\ldots \circ F_2\circ F_1, \qquad {F(0):=\operatorname{id}.}$$
If all maps $F_n$ have an attracting fixed point at the origin, we can then ask about the properties of the non-autonomous basin of attraction
$$
\Omega:=\{p\in \C^2\mid F(n)(p)\xrightarrow{n\rightarrow\infty} 0\}.
$$
In general, non-autonomous basins need not be open \cite[Theorem 6]{PW}. Moreover, it can happen that $\Omega$ has empty interior. It is therefore natural to impose additional conditions on the sequence $(F_n)$. For example, we can assume that the sequence $(F_n)$ is uniformly attracting at the origin, that is, there exist $0<C<D<1$ and $R>0$ such that 
$$C\|p\|\leq \|F_n(p)\|\leq D\|p\|, \qquad \forall p\in \mathbb{B}_R.$$
In this case it is easy to see that $\Omega=\bigcup_{n\geq0}F(n)^{-1}(\mathbb{B}_R)$ is an open set. It is therefore natural to ask about its complex structure. {In the autonomous case, after a fixed linear change of coordinates and restriction to a sufficiently small ball, the uniform attraction condition holds, and} we know that $\Omega\simeq\C^2$. Hence one may ask whether the same holds for every uniformly attracting non-autonomous sequence. 
This question is usually referred to as Bedford's conjecture. It was proved under several additional assumptions, see \cite{AAM,FS,Peters,PS, Wold}. In particular, Wold obtained the conclusion under the condition $D^2<C$. The conjecture was {settled} by Bera and Verma \cite{BV}, who proved that every uniformly attracting non-autonomous basin is biholomorphic to $\C^2$.

A natural generalization of the uniformly attracting case is obtained by relaxing either the upper or the lower bound. By allowing $D=1$, one can easily construct a sequence of linear automorphisms $F_n$ for which $\Omega=\{0\}$, {so a uniform upper contraction bound is useful for ensuring that the basin is open.} By contrast, if we retain the uniform upper bound $D<1$ and drop the positive lower bound, we always have $\Omega=\bigcup_{n\geq0}(F(n))^{-1}(\mathbb{B}_R)$, so the basin remains open, but new phenomena can occur. Indeed, given a sequence of H\'enon maps
\begin{equation}\label{family}
    F_n(z,w)=(z^2+a_nw,a_nz),
\end{equation}
where $0<|a_n|<1$, Forn\ae ss proved the following result, see \cite{Fshort}.

\begin{theorem}\label{short} If $|a_{n+1}|\leq|a_n|^2$ for all $n\geq {1}$, then the resulting non-autonomous basin $\Omega$ has the following properties:
 \begin{enumerate}[label=(\roman*)]
     \item $\Omega$ is an increasing union of domains, each biholomorphic to a ball.
     \item The infinitesimal Kobayashi metric of $\Omega$ vanishes identically.
     \item There exists a nonconstant bounded plurisubharmonic function defined on $\Omega$.
 \end{enumerate}
    
\end{theorem}

In particular, property $(iii)$ implies that $\Omega$ is not biholomorphic to $\C^2$. A domain satisfying these three properties is called a \emph{Short $\C^2$}. By contrast, Forn\ae ss also proved the following.

\begin{theorem}\label{c2} Given a sequence of maps \eqref{family} where $0<|a_n|<b<1$ and $|a_{n+1}|\geq|a_n|^t$ for some fixed  $1<t<2$ and for all $n\geq {1}$, its non-autonomous basin $\Omega$ is biholomorphic to $\C^2$.  
    
\end{theorem}

Since properties $(i)$ and $(ii)$ remain valid for any non-autonomous system with a uniform upper contraction bound, it is natural to ask whether such a domain must either be biholomorphic to $\C^2$ or be a Short $\C^2$, or whether other possibilities can occur. 
\medskip

To understand the gap between these two results, we first rewrite their hypotheses in a common form. {From this point on, we take $0<a_n<1$ for $n\geq1$ and assume that $a_n\to0$.} If we fix $0<c<1$, then we can choose a sequence of positive real numbers $(q_n)$ satisfying $q_n\rightarrow \infty$ such that $a_n=c^{q_n}$ for all $n\geq0$.
We define  
$$\rho_n:=q_n/2^n.$$
The assumption on $(a_n)$ in Theorem \ref{short} can now be written as
\begin{equation}\label{eq:shortC2} q_{n+1}\geq 2q_n \quad\text{ or equivalently }\quad \rho_{n+1}\geq\rho_n.
\end{equation}

Note that $q_{n+1}\geq {2^{n}q_1}$  and that  the sequence $(\rho_n)$ is non-decreasing, and hence remains bounded away from $0$.  
\medskip 

For some fixed $1<t<2$, the assumption on $(a_n)$ in Theorem \ref{c2} can now be written as
\begin{equation}\label{eq:C2} q_{n+1}\leq tq_n \quad\text{ or equivalently }\quad 
    \rho_{n+1}\leq\frac{t}{2}\rho_n.
\end{equation}
Thus $q_{n+1} \leq {t^{n}q_1}$ and the sequence $(\rho_n)$ converges to $0$. 
\medskip

{These results point to $2^n$ as the natural growth scale and motivate the critical condition $q_n^{1/n}\to2$, where potentially new  complex structures may appear. Forn\ae ss and Pal \cite[Theorem~1.10]{FP} constructed sequences on this scale whose basins are Short $\C^2$ and sequences whose basins are biholomorphic to $\C^2$. }
This leads to the following question, proposed by J.~E. Forn\ae ss. 
\begin{question}\label{q:original} Let $0<c<1$.  Does there exist a sequence of real numbers $(q_n)$ satisfying  $\sqrt[n]{q_n}\rightarrow 2$, for which the non-autonomous basin associated with
$$ F_n(z,w)=(z^2+c^{q_n}w,c^{q_n}z),$$
 is neither biholomorphic to $\C^2$ nor a Short $\C^2$? 
\end{question}

{We give criteria for non-existence of bounded non-constant plurisubharmonic functions and for biholomorphism with $\C^2$, and study the normalization by $DF(n)(0)$. The condition $q_{n+1}/q_n\to2$ implies $q_n^{1/n}\to2$, but the converse need not hold. Under the quotient condition, we characterize when the basin is Short $\C^2$. We also construct a non-autonomous basin, which supports no bounded non-constant plurisubharmonic functions, and for which the Fatou--Bieberbach domain obtained from derivative normalization is a proper subset of the basin. We do not determine whether the full basin is biholomorphic to $\C^2$.}

\subsection{Statement of results}

We say that a Stein manifold $X$ is \emph{parabolic} if every bounded-above plurisubharmonic function on $X$ is constant. In particular, such manifolds are not Short $\C^2$. {For the H\'enon basins considered below, the other two defining properties of Short $\C^2$ hold by \Cref{prop:basic}. Thus parabolicity isolates the obstruction relevant to \Cref{q:original}.} 
\medskip

Let $(F_j)$ be a sequence of automorphisms of $\C^2$ satisfying 
$$a_j\|p\|\leq \|F_j(p)\|\leq b_j\|p\|, \qquad \forall p\in \mathbb{B}_R,$$
where $0<a_j<  b_j\leq b<1$ for all $j>0$.
Our first result provides sufficient conditions that prevent the non-autonomous basin of $(F_n)$ from being a Short $\C^2$.

\begin{theorem}\label{notshort} If there exists an integer $m>1$ such that 
$$b_j^m<a_j\qquad \forall j>0,$$ 
then $\Omega$ is parabolic. Moreover if $${\prod_{i=1}^{j-1}b_i\leq b_j^m}$$ for all sufficiently large $j$ then $\Omega\simeq\C^2$.
\end{theorem}
In contrast to the uniformly attracting case, the theorem allows sequences for which  $\liminf a_j= 0$. The proof of the first statement relies on potential theory, in particular on estimates for relative extremal functions, whereas the additional assumption in the second statement enables us to translate the problem to the uniformly attracting case. {Parabolic Stein manifolds and relative extremal functions were studied in \cite{Zahariuta,AS}; see in particular \cite[Definition~1 and Proposition~1]{AS}.}

The parabolicity assertion is a consequence of the stronger cumulative criterion in \Cref{thm:cumulative-parabolic}: with $A_n=\prod_{j=1}^n a_j$ and $B_n=\prod_{j=1}^n b_j$, it suffices that $B_n^m\le A_n$ along an unbounded set of indices.

\begin{remark}
The notion of parabolicity is defined for Stein manifolds, so we first verify that $\Omega$ is Stein. If $F_n$ are global automorphisms of $\C^2$ satisfying $\|F_n(p)\|\leq b\|p\|$ on $\B_R$ for some uniform $0<b<1$, then $\Omega=\bigcup_{n\geq0}F(n)^{-1}(\mathbb{B}_R)$. The sets $F(n)^{-1}(\mathbb{B}_R)$ are Runge in $\C^2$ and exhaust $\Omega$. It follows that $\Omega$ is Stein.
\end{remark}

\medskip

{We now turn to our H\'enon family with positive coefficients}
\begin{equation}\label{family2}
    F_n(z,w)=(z^2+a_nw,a_nz),
\end{equation}
{Here $a_n=c^{q_n}$, $0<c<1$, and the positive real numbers $q_n$ tend to infinity. These assumptions hold throughout our analysis of this family. We use}
$$
\rho_n:=\frac{q_n}{2^n},
\qquad
d_n:=2q_n-q_{n+1}.
$$
The quantity $\rho_n$ measures the growth of $q_n$ relative to
the scale $2^n$, while $d_n$ measures the deviation from the
relation $q_{n+1}=2q_n$. We begin with two simple criteria.

\begin{samepage}
\begin{theorem}\label{thm:first-classification} 
The following statements hold.
\begin{enumerate}[label=\textup{(\roman*)}]
\item If $\liminf \rho_n>0$, then $\Omega$ is a Short $\C^2$.
\item If $\liminf \rho_n=0$ and $\liminf d_n>0$, then $\Omega$ is
biholomorphic to $\C^2$.
\end{enumerate}
\end{theorem}
\end{samepage}
{The first statement is a consequence of the general construction of Short $\C^2$ domains of Forn\ae ss \cite[Theorem~1.4]{Fshort}; see also the discussion preceding Theorem~1.7 in \cite{FP}. Indeed, $\liminf\rho_n>0$ gives $a_n\le b^{2^n}$ eventually for some $0<b<1$. We include a direct proof adapted to the present family. The second statement extends the earlier $\C^2$ criterion \eqref{eq:C2} to the critical case. For example, it applies to $q_n=2^n/n$.}
\medskip

We also provide a sufficient condition on the family of maps \eqref{family2} that prevents $\Omega$ from being a Short $\C^2$. 

\begin{theorem}\label{thm:second-classification} If $\liminf \rho_n=0$ and if there exists an integer $m>1$ {such that} 
\begin{equation}\label{eq:henon-parabolic}
 {\prod_{j=1}^{n-1}a_j\geq a_n^m}   
\end{equation}
for all sufficiently large $n$, then  $\Omega$ is parabolic. In particular, $\Omega$ is not a Short $\C^2$.
\end{theorem}

{When $q_{n+1}/q_n\to2$, the product condition \eqref{eq:henon-parabolic} is automatic} by \Cref{lem:Q-ratio}. Combining the preceding criteria gives the following characterization, proved in \Cref{cor:regular-short}.
\begin{corollary}\label{cor:intro-regular}
Assume $q_{n+1}/q_n\to2$. Then $\Omega$ is a Short $\C^2$ if and only if $\liminf\rho_n>0$. {In particular, if $\liminf\rho_n=0$, then $\Omega$ is parabolic.}
\end{corollary}

{Since $\Omega$ is Stein, \cite[Theorem~1]{AS} also implies that in the parabolic case $\mathcal O(\Omega)$ and $\mathcal O(\C^2)$ are isomorphic as Fr\'echet spaces. However, this does not determine whether the manifolds are biholomorphic.} 
\medskip

We normalize the compositions $F(n)$ of the H\'enon maps  by their derivative at the origin and set
$$
\Phi_n:=(DF(n)(0))^{-1}\circ F(n).
$$
Derivative normalization of the orbit is typically used in the uniformly attracting setting. In particular it provides a biholomorphism between the $\C^2$ and the region of the uniform convergence of such normalizations.

The maps $\Phi_n$ satisfy an exact recurrence in which the successive factors
alternate between quadratic shears
$$\sigma_{2k}(z,w)=(z,w+\lambda_{2k}z^2),\qquad \sigma_{2k+1}(z,w)=(z+\lambda_{2k+1}w^2,w)$$
where
$$
{\lambda_n:=\frac{a_1\cdots a_n}{a_{n+1}},\qquad n\ge0,\quad\text{with }\lambda_0:=\frac{1}{a_{1}}.}
$$

Thus the nonlinear part of the normalized dynamics is encoded
by the sequence $(\lambda_n)$. The convergence or divergence of
\begin{equation}\label{eq:sum}
\sum_{n=0}^{\infty}\lambda_n
\end{equation}
provides a first natural division in our analysis. If this series converges,
then the inverse maps $\Phi_n^{-1}$ converge locally uniformly on $\C^2$ to
an injective entire map $\Psi:\C^2\longrightarrow\Omega$, whose image $\mathcal U:=\Psi(\C^2)$
is a Fatou--Bieberbach subdomain of $\Omega$. The main question in the
summable case is therefore whether $\mathcal U$ coincides with the whole
basin $\Omega$.
\medskip  

We now assume that the sum \eqref{eq:sum}  converges and set
$$
L_n:=\sum_{j=n}^{\infty}\lambda_j.
$$
For all sufficiently large $n$ we have $0<\lambda_n<L_n<1$, and we define
$$
\alpha_n:=
\frac{\log(1/L_n)}{\log(1/\lambda_n)}< 1.
$$

\begin{theorem}\label{thm:psh-capture}
{Assume $\sum_n\lambda_n<\infty$ and $\liminf\rho_n=0$. If $\liminf\alpha_n>0$,} then $\Omega=\mathcal U$ and hence $\Omega$ is biholomorphic to $\C^2$.
\end{theorem}

\begin{remark} The assumption $\liminf d_n>0$ in part $(ii)$ of Theorem \ref{thm:first-classification} implies that \eqref{eq:sum} is summable and that $\alpha_n\rightarrow 1$. Hence, part $(ii)$ will follow from Theorem \ref{thm:psh-capture}. 
\end{remark}

\medskip

The stronger capture criterion in \Cref{thm:alpha-capture} uses lower bounds for $\alpha_n$ only over suitable intervals. The following example, proved in \Cref{thm:proper-core-example}, shows that $\Ucal$ can be a proper subset of $\Omega$ {even when $q_{n+1}/q_n\to2$}.
\begin{theorem}\label{thm:intro-proper}
For every $0<c<1$, there is a positive sequence $(q_n)$ such that
$$
q_n^{1/n}\to2,\qquad q_{n+1}/q_n\to2,\qquad \rho_n\to0,
$$
and
$$
\sum_n\lambda_n<\infty,\qquad \alpha_n\to0,
$$
for which $\Omega$ is parabolic and $\Ucal\subsetneq\Omega$. Moreover, $\Omega\setminus\Ucal$ contains three pairwise disjoint holomorphic disks.
\end{theorem}
This particular proper inclusion does not exclude biholomorphism $\Omega\simeq\C^2$. In the case where \eqref{eq:sum} is nonsummable, the positive quadratic coefficients prevent $\Phi_n$ from converging locally uniformly on any neighborhood of the origin; the construction of $\Ucal$ used here is therefore unavailable. 

\smallskip
\noindent\emph{Organization.}
Section~\ref{sec:general} proves the general contraction criteria. Sections~\ref{sec:basic}--\ref{sec:liouville} establish the H\'enon estimates and parabolicity criterion. {Section~\ref{sec:summable} constructs the domain $\Ucal$ and proves \Cref{thm:psh-capture}; Section~\ref{sec:proper-core} gives an example with $\Ucal\subsetneq\Omega$.}

\smallskip
\noindent{\bf AI Disclaimer: }
The main ideas and proofs leading to the parabolicity criterion in Theorem~1.4, Theorems~1.6 and~1.7, Corollary~1.8, and Theorem~1.9 were already contained in the first version of the manuscript and were developed and written by the author without AI assistance.

ChatGPT was subsequently used as an exploratory tool in developing further extensions of these results. In particular, the sufficient condition for biholomorphism with \(\mathbb C^2\) in Theorem~1.4 (proved as Theorem~2.6), the refinement in Theorem~6.6, and the results of Section~7 were developed through iterative discussions with ChatGPT involving the exploration of possible arguments, the examination of intermediate claims, and the refinement of proofs.

ChatGPT was also used for language editing and for improving the clarity and exposition of the text.

\section{General contraction criteria for non-autonomous basins}
\label{sec:general}

Throughout this section let $F_j\in\operatorname{Aut}(\C^2)$ and assume that
\begin{equation}\label{eq:general-bounds}
a_j\|p\|\leq \|F_j(p)\|\leq b_j\|p\|,
\qquad p\in\B,
\end{equation}
where $0<a_j\leq b_j\leq b<1$.
Here and below $\B$ denotes the unit Euclidean ball in $\C^2$.
There is no loss of generality in taking the radius equal to one, since
a fixed linear rescaling preserves all the hypotheses.

We write
$$
A_n:=\prod_{j=1}^n a_j,\qquad
B_n:=\prod_{j=1}^n b_j,
\qquad A_0=B_0:=1.
$$
Observe that for any $0<r\leq 1$ we have $\Omega =\bigcup_{n>0}F(n)^{-1}(\B_r)$. Moreover, \eqref{eq:general-bounds} gives 
\begin{equation}\label{eq:product-bounds}
A_n\|p\|\leq \|F(n)(p)\|\leq B_n\|p\|,
\qquad p\in\B_r.
\end{equation}

\subsection{Relative extremal functions}
Suppose that $D$ is an open subset of $\C^n$ and $E\subset D$. The \emph{relative extremal function} for $E$ in $D$ is defined by
$$
u_{E,D}(z):= \sup\left\{v(z):v\in\PSH(D),\;v\leq0,\;v|_E\leq-1\right\},
$$
and its upper semicontinuous regularization $u_{E,D}^*$ is plurisubharmonic on $D$. These functions will play an important role in the proofs of Theorems \ref{notshort} and \ref{thm:second-classification}. We shall only need the following explicit form for two concentric balls.

For $0<r<1$, we have 
\begin{equation}\label{eq:extremal function}
u_{\overline{\B}_r,\B}(z)= \max\left\{-1,\, \frac{\log\|z\|}{\log(1/r)} \right\}.
\end{equation}
This function is continuous, and therefore  $u_{\overline{\B}_r,\B}^* =u_{\overline{\B}_r,\B}$. For further details, see \cite[Chapter 4.5]{Klimek}. 

\medskip

We now fix $0<r<1$ and define the sequence of domains $\Omega_n:=F(n)^{-1}(\B)$. Since $F_j(\B)\subset\B$, the domains $\Omega_n$ are increasing and
$\B\subset\Omega_n$ for every $n$. In particular, they exhaust $\Omega$.

\begin{lemma}\label{lem:extremal-exhaustion}
We have $u^*_{\overline{\B}_r,\Omega_n} \searrow u^*_{\overline{\B}_r,\Omega}$ pointwise on $\Omega$.
\end{lemma}

\begin{proof}
By monotonicity of the relative extremal function with respect to the
ambient domain,
$$
u^*_{\overline{\B}_r,\Omega_{n+1}}
\leq
u^*_{\overline{\B}_r,\Omega_n}
$$
on $\Omega_n$. Hence $u:=\lim_{n\to\infty}u^*_{\overline{\B}_r,\Omega_n}$
is well defined on $\Omega$. We first note that
$$
u^*_{\overline{\B}_r,\Omega_n}=-1
\qquad\text{on }\overline{\B}_r.
$$
Indeed, choose $r<R<1$. Since $\B_R\subset\Omega_n$,
monotonicity and the explicit formula for concentric balls give
$$
-1
\leq
u^*_{\overline{\B}_r,\Omega_n}
\leq
u^*_{\overline{\B}_r,\B_R}
=-1
\qquad\text{on }\overline{\B}_r.
$$

Thus $-1\leq u\leq0$ and $u=-1$ on $\overline{\B}_r$.
Since $u$ is the decreasing limit of plurisubharmonic
functions, it follows that $u\in\PSH(\Omega)$, and therefore
$$
u\leq u^*_{\overline{\B}_r,\Omega}.
$$

Conversely, for every $n>0$, the restriction of $u_{\overline{\B}_r,\Omega}$ to $\Omega_n$ satisfies
$$
(u^*_{\overline{\B}_r,\Omega})|_{\Omega_n}
\leq
u^*_{\overline{\B}_r,\Omega_n}.
$$
Letting $n\to\infty$ gives $u^*_{\overline{\B}_r,\Omega}\leq u$, and the result follows.
\end{proof}

\begin{lemma}\label{lem:extremal-visibility}
Suppose that there are numbers $0<\beta_n<1$ such that
$
\overline{\B}_{\beta_n}
\subset
F(n)(\overline{\B}_r).
$
Assume moreover that there exists $0<\eta<1$ such that, for every
$p\in\Omega$, there are arbitrarily large $n$ for which
$$
\|F(n)(p)\|\leq\beta_n^\eta.
$$
Then $u^*_{\overline{\B}_r,\Omega}\leq-\eta$ on $\Omega$.
\end{lemma}

\begin{proof}
Since $F(n):\Omega_n\longrightarrow\B$ is biholomorphic, invariance and monotonicity of the relative extremal
function give
$$
u^*_{\overline{\B}_r,\Omega_n}(p)
=
u^*_{F(n)(\overline{\B}_r),\B}(F(n)(p))
\leq
u^*_{\overline{\B}_{\beta_n},\B}(F(n)(p)).
$$
By \eqref{eq:extremal function} we have 
$$
u^*_{\overline{\B}_{\beta_n},\B}(z)
=
\max\left\{
-1,\,
\frac{\log\|z\|}{\log(1/\beta_n)}
\right\}.
$$
Thus $\|F(n)(p)\|\leq\beta_n^\eta$ implies
$u^*_{\overline{\B}_r,\Omega_n}(p)\leq-\eta.$
By assumption, for every fixed $p\in\Omega$ this inequality holds for arbitrarily large $n$. Therefore, by Lemma \ref{lem:extremal-exhaustion}, we obtain $u^*_{\overline{\B}_r,\Omega}(p)\leq-\eta$.
\end{proof}

\begin{lemma}\label{lem:negative-extremal-parabolic}
Suppose that for some $0<r<1$ and some $\eta>0$, we have
$$
u^*_{\overline{\B}_r,\Omega}\leq-\eta
$$
on $\Omega$. Then every bounded-above plurisubharmonic function on
$\Omega$ is constant.
\end{lemma}

\begin{proof}
Let $v\in\PSH(\Omega)$ be bounded above and assume that
$v\not\equiv-\infty$. Put
$$
M:=\sup_\Omega v,
\qquad
A:=\sup_{\overline{\B}_r}v.
$$
If $A=M$, then $v$ attains its maximum at an interior point of
$\Omega$, and hence is constant.

Suppose that $A<M$. Then the function
$$
w:=\frac{v-M}{M-A}
$$
is plurisubharmonic on $\Omega$ and satisfies
$$
w\leq0\quad\text{on }\Omega,
\quad \text{ and } \quad
w\leq-1\quad\text{on }\overline{\B}_r.
$$
It follows that
$$
w\leq
u_{\overline{\B}_r,\Omega}
\leq
u^*_{\overline{\B}_r,\Omega}
\leq-\eta<0.
$$
This contradicts $\sup_\Omega w=0$, hence $v$ is constant.
\end{proof}

\subsection{A criterion for parabolicity}
Our first result gives a sufficient condition for $\Omega$ to be parabolic. 
\begin{theorem}\label{thm:cumulative-parabolic}
Assume that for some integer $m>1$ there exist arbitrarily large
integers $n$ such that
\begin{equation}\label{eq:cumulative-power}
B_n^m\leq A_n.
\end{equation}
Then the non-autonomous basin $\Omega$ is parabolic.
\end{theorem}

\begin{proof}
Fix $0<r<1$. We apply Lemmas
\ref{lem:extremal-exhaustion},
\ref{lem:extremal-visibility}, and
\ref{lem:negative-extremal-parabolic} with the fixed compact set
$\overline{\B}_r$.

By \eqref{eq:product-bounds},
$$
A_n\|p\|\leq \|F(n)(p)\|,
\qquad p\in\B_r.
$$
Since $F(n)$ is an automorphism fixing the origin, it follows that $
\overline{\B}_{rA_n}
\subset
F(n)(\overline{\B}_r).
$
Thus, in Lemma \ref{lem:extremal-visibility}, we may take $\beta_n:=rA_n$.
\medskip

Fix $p\in\Omega$. Since $F(n)(p)\to0$, there exists $s>0$ such that
$F(s)(p)\in\B$. Hence, for $n>s$,
$$
\|F(n)(p)\|
\leq
\frac{B_n}{B_s}\|F(s)(p)\|
=
C_pB_n,
$$
for some constant $C_p>0$.

Now consider the arbitrarily large indices $n$ for which $B_n^m\leq A_n$, and set $\eta:=\frac1{2m}$.
For such $n$,
$$
\beta_n^\eta
=
(rA_n)^\eta
\geq
r^\eta B_n^{m\eta}
=
r^\eta B_n^{1/2}.
$$
Since $B_n\leq b^n\to0$, we have, for all sufficiently large such
indices,
$$
C_pB_n
\leq
r^\eta B_n^{1/2}
\leq
\beta_n^\eta.
$$
Therefore $\|F(n)(p)\|\leq\beta_n^\eta$. Lemma \ref{lem:extremal-visibility}, together with
Lemma \ref{lem:extremal-exhaustion}, now gives
$$
u^*_{\overline{\B}_r,\Omega}
\leq
-\frac1{2m}
$$
on $\Omega$. By Lemma
\ref{lem:negative-extremal-parabolic}, every bounded-above
plurisubharmonic function on $\Omega$ is constant, hence $\Omega$ is parabolic.
\end{proof}

The following corollary is an immediate consequence of Theorem \ref{thm:cumulative-parabolic} and yields the first assertion of Theorem \ref{notshort}. 

\begin{corollary}\label{cor:pointwise-parabolic}
Suppose that there exists an integer $m>1$ such that
\begin{equation}\label{eq:pointwise-power}
b_j^m\leq a_j
\end{equation}
for all $j>0$. Then $\Omega$ is parabolic.
\end{corollary}

\subsection{A sufficient condition for \texorpdfstring{$\Omega\simeq\C^2$}{Omega=C2}}
We next give an additional condition under which the non-autonomous basin
from \Cref{cor:pointwise-parabolic} is in fact biholomorphic to
$\C^2$. This will prove the second assertion of Theorem \ref{notshort}.

\begin{theorem}\label{thm:general-Ck}
Assume that for some integer $m>1$,
\begin{equation}\label{eq:general-first-condition}
b_j^m<a_j
\end{equation}
for every $j>0$, and that
\begin{equation}\label{eq:general-second-condition}
{\prod_{i=1}^{j-1}b_i\leq b_j^m}
\end{equation}
for all sufficiently large $j$. Then $\Omega\simeq\C^2$.
\end{theorem}

\begin{proof}
Let $j_0$ be such that \eqref{eq:general-second-condition} holds for
every $j\geq j_0$, and put
$$
B_n:=\prod_{i=1}^n b_i,
\qquad
B_0:=1,
\qquad
\theta:=\frac{m-1}{m}<1.
$$
Thus, for every $j\geq j_0$, we have $B_{j-1}\leq b_j^m$. For $0<\varepsilon<1/2$, to be specified below, set
$$
r_n:=\varepsilon B_n^\theta,
\qquad \forall n\geq j_0-1.
$$
Then $r_j=r_{j-1}b_j^\theta$. Since $0<b_j<1$ and $\theta<1$, we have $b_j\leq b_j^\theta$, and
therefore, by \eqref{eq:general-bounds},
$$
F_j(\B_{r_{j-1}})
\subset
\B_{b_jr_{j-1}}
\subset
\B_{r_j},
\qquad j\geq j_0.
$$

Let $\Omega^{(j_0)}$ be the non-autonomous basin of the sequence $F_{j_0},F_{j_0+1},\ldots$.
Deleting finitely many maps does not change the biholomorphism type
of the basin, so $\Omega\simeq\Omega^{(j_0)}$. For $j\geq j_0$ define
$$
\widehat F_j(z)
:=
\frac1{r_j}F_j(r_{j-1}z).
$$
The basin $\widehat\Omega$ of
$(\widehat F_j)_{j\geq j_0}$ is biholomorphic to $\Omega^{(j_0)}$. Indeed, let $\Lambda_n(z):=\frac{z}{r_n}$ and observe that
$$
\widehat F_j
=
\Lambda_j\circ F_j\circ\Lambda_{j-1}^{-1},
$$
and hence
$$
\widehat F_n\circ\cdots\circ\widehat F_{j_0}
=
\Lambda_n\circ
F_n\circ\cdots\circ F_{j_0}
\circ\Lambda_{j_0-1}^{-1}.
$$
If $p\in\Omega^{(j_0)}$, choose $s\geq j_0$ such that $F_s\circ\cdots\circ F_{j_0}(p)\in\B$.
Then, for some constant $C_p>0$ and all $n>s$, we have
$$
\left\|F_n\circ\cdots\circ F_{j_0}(p)\right\|
\leq C_p B_n.
$$
Consequently,
$$
\frac{
\left\|F_n\circ\cdots\circ F_{j_0}(p)\right\|
}{r_n}
\leq
\frac{C_p}{\varepsilon}B_n^{1-\theta}
\longrightarrow0.
$$
Thus $\Lambda_{j_0-1}(p)\in\widehat\Omega$.

Conversely, if the rescaled orbit tends to the origin, then clearly
$$
F_n\circ\cdots\circ F_{j_0}(p)
=
r_n\,
\widehat F_n\circ\cdots\circ\widehat F_{j_0}
\bigl(\Lambda_{j_0-1}(p)\bigr)
\longrightarrow0,
$$
hence $\widehat\Omega=\Lambda_{j_0-1}\bigl(\Omega^{(j_0)}\bigr)$. 
\medskip

It therefore remains to prove that $\widehat\Omega\simeq\C^2$. We do this by decomposing each map $\widehat F_j$ into finitely many maps that together form a uniformly attracting sequence, to which the theorem of Bera--Verma applies.
\medskip

Let
$$
L_j:=DF_j(0).
$$
Differentiating \eqref{eq:general-bounds} at the origin gives
$$
a_j\|v\|
\leq
\|L_jv\|
\leq
b_j\|v\|.
$$
Moreover, since $\sup_{\B}\|F_j\|\leq b_j$, the Cauchy estimates give a constant $C>0$, independent of $j$, such
that
$$
\|F_j(w)-L_jw\|
\leq
Cb_j\|w\|^2,
\quad \text{ for all }\quad
\|w\|\leq\frac{1}{2}.
$$

We next define 
$$
G_j(z)
:=
D\widehat F_j(0)^{-1}\widehat F_j(z)
=
\frac1{r_{j-1}}L_j^{-1}F_j(r_{j-1}z).
$$
Then $G_j(0)=0$ and $DG_j(0)=I$. Since $r_{j-1}<1/2$, we obtain, for $z\in\B$,
$$
\|G_j(z)-z\|
\leq
C\frac{b_j}{a_j}r_{j-1}\|z\|^2.
$$
Now \eqref{eq:general-first-condition} and
\eqref{eq:general-second-condition} give
$$
\frac{b_j}{a_j}r_{j-1}
<
\varepsilon b_j^{1-m}B_{j-1}^\theta
\leq
\varepsilon b_j^{1-m}(b_j^m)^\theta
=
\varepsilon,
$$
because $\theta=\frac{m-1}{m}$. Hence
$$
\|G_j(z)-z\|
\leq
\delta\|z\|^2,
\qquad z\in\B,
$$
where $\delta:=C\varepsilon$. Recall that $b_j\leq b<1$ for all $j$, and set
$\beta:=b^{1-\theta}<1$. Choose $\beta<\tau<1$ and then fix $\varepsilon$ sufficiently small so that 
$$
\delta<1,
\qquad
\tau(1+\delta)<1.
$$
Define
$$V_j:=\tau G_j,\qquad R_j:= \frac{r_{j-1}}{\tau r_j}L_j=
\frac{b_j^{-\theta}}{\tau}L_j.
$$ 
Observe that
$$
\widehat F_j=R_j\circ V_j,
$$
and
$$
\tau(1-\delta)\|z\|
\leq
\|V_j(z)\|
\leq
\tau(1+\delta)\|z\|,
\qquad z\in\B.
$$
In particular, the maps $V_j$ are uniformly attracting. Let $\sigma_j^+$ and $\sigma_j^-$ denote the largest and smallest
singular values of $R_j$. Then
\begin{equation}\label{eq:sigma+}
 \sigma_j^+
\leq
\frac{b_j^{1-\theta}}{\tau}
\leq
\frac{\beta}{\tau}
=:\rho<1,   
\end{equation}
whereas, by \eqref{eq:general-first-condition},
$$
\sigma_j^-
\geq
\frac{a_jb_j^{-\theta}}{\tau}
>
\frac{b_j^{m-\theta}}{\tau}.
$$

Set $K_0:=\frac{m-\theta}{1-\theta}$. Equation \eqref{eq:sigma+} gives $b_j\geq
(\tau\sigma_j^+)^{1/(1-\theta)}$
and hence
$$
\sigma_j^->\tau^{K_0-1}(\sigma_j^+)^{K_0}.
$$
Since $\sigma_j^+\leq\rho<1$, we may choose $K>K_0$, independent of
$j$, so large that $\rho^{K-K_0}\leq\tau^{K_0-1}$. It follows that
$$
(\sigma_j^+)^K<\sigma_j^-.
$$

We now factor the linear maps $R_j$ into uniformly contracting
factors. For each $j$, choose an integer $N_j\geq1$ such that
$$
\rho^{N_j+1}<\sigma_j^+\leq \rho^{N_j}.
$$
Let $R_j=U_j\circ H_j$ be the polar decomposition, where $U_j$ is unitary and $H_j$ is
positive definite, and put
$$
Q_j:=H_j^{1/N_j}.
$$
Then
$$
R_j
=
(U_jQ_j)\circ
\underbrace{Q_j\circ\cdots\circ Q_j}_{N_j-1\ {\rm times}}.
$$
The largest singular value of each of these linear factors is $(\sigma_j^+)^{1/N_j}\leq \rho$.
On the other hand,
$$
(\sigma_j^-)^{1/N_j}
>
(\sigma_j^+)^{K/N_j}
>
\rho^{K(N_j+1)/N_j}
\geq \rho^{2K}.
$$
Thus every linear factor satisfies
$$
\rho^{2K}\|z\|
<
\|Q_jz\|
\leq
\rho\|z\|.
$$
The same estimates hold for $U_jQ_j$, since $U_j$ is unitary.

We now replace each map
$$
\widehat F_j=R_j\circ V_j
$$
by the finite block consisting of $V_j$ followed by the above linear
factors of $R_j$. If
$$
A:=
\min\{\rho^{2K},\,\tau(1-\delta)\},
\qquad
B:=
\max\{\rho,\,\tau(1+\delta)\},
$$
then $0<A<B<1$ and all the maps in the refined decomposition form a uniformly attracting sequence with respect to these bounds. {Its non-autonomous basin coincides with $\widehat\Omega$. Indeed, at the ends of the inserted blocks the compositions are exactly those of the rescaled sequence. Conversely, once those endpoint values lie in $\B$ and tend to zero, every intermediate value has norm at most the preceding endpoint norm, since all inserted factors contract $\B$.  The conclusion now follows from \cite[Theorem~1.1]{BV}.} Thus
$\widehat\Omega\simeq\C^2$,
and hence $\Omega\simeq\C^2$.
\end{proof}

\section{Basic properties of the non-autonomous basin for a H\'enon family}\label{sec:basic}
From now on we shall restrict to our H\'enon family
$$
 F_n(z,w)=(z^2+a_nw,a_nz),
 \qquad
 a_n=c^{q_n},\quad 0<c<1,
$$
with $q_n\to\infty$. We use the Euclidean norm unless stated otherwise.

\begin{proposition}\label{prop:basic}
The non-autonomous basin $\Omega$ is a connected, proper open subset of $\C^2$. It is an increasing union of domains biholomorphic to a ball, and its infinitesimal Kobayashi metric vanishes identically. Moreover, $F(n)\to0$ locally uniformly on $\Omega$.
\end{proposition}

\begin{proof}
Since $a_n\to0$, after deleting finitely many terms we may assume $a_n\le1/4$ for every $n$. For $p\in\B_{1/4}$, a direct estimate gives $\norm{F_n(p)}\le\frac34\norm{p}$. Hence $\B_{1/4}$ is forward invariant and $F(n)\to0$ uniformly there. The domains  $\Omega_n:=F(n)^{-1}(\B_{1/4})$
are increasing, each is biholomorphic to a ball, and their union is $\Omega$. This proves openness, connectedness, and local uniform convergence to zero on $\Omega$.
\medskip

Fix $p\in\Omega$ and write $(z_n,w_n):=F(n)(p)$. Then
$$
 DF_{n+1}(F(n)(p))=
 \begin{pmatrix}2z_n&a_{n+1}\\ a_{n+1}&0\end{pmatrix}
 \longrightarrow0.
$$
Thus, for sufficiently large $n$, the derivatives along the orbit have norm at most $1/2$, and consequently $DF(n)(p)v\to0$ for every $v\in T_p\C^2$. By biholomorphic invariance and monotonicity of the Kobayashi metric,
$$
 \kappa_\Omega(p;v)
 \le\kappa_{\Omega_n}(p;v)
 =\kappa_{\B_{1/4}}(F(n)(p);DF(n)(p)v)\longrightarrow0.
$$
Finally, for every real $R>1$ the first coordinate of the orbit of $(R,R)$ is positive and strictly increasing, hence this orbit escapes to infinity. Thus $\Omega\ne\C^2$.
\end{proof}

{For completeness we give a direct proof of the following consequence of Forn\ae ss's  criterion \cite[Theorem~1.4]{Fshort}; see also \cite[Introduction]{FP}. It provides the first assertion of \Cref{thm:first-classification}.}

\begin{theorem}\label{thm:short}
If $\liminf_{n\to\infty}\rho_n>0$, then $\Omega$ is a Short $\C^2$.
\end{theorem}

\begin{proof}
Choose $0<\kappa<2\liminf_{n\to\infty}\rho_n$. After deleting finitely many terms, we have
$q_{n+1}\ge \kappa 2^n$ for all sufficiently large $n$. Set $\beta_n:=c^{\kappa2^n}$, and observe that $a_{n+1}\le\beta_n$ and $\beta_{n+1}=\beta_n^2$ for all sufficiently large $n$.
For $p\in\C^2$ put
$$
 M_n(p):=\max\bigl\{\norm{F(n)(p)}_\infty,\beta_n\bigr\}.
$$
The definition of $F_{n+1}$ gives $ M_{n+1}\le2M_n^2$. Indeed, $a_{n+1}\le M_n$ and both coordinates of $F(n)(p)$ have modulus at most $M_n$.

Define
$$
 u_n:=2^{-n}\log M_n+
 \sum_{j=n+1}^{\infty}2^{-j}\log2.
$$
Each $u_n$ is plurisubharmonic and the preceding inequality implies $u_{n+1}\le u_n$. Hence $u_n$ decreases to a plurisubharmonic function $u$ that is not identically $-\infty$ because
$$
 u(0)=\kappa\log c.
$$
For $p\in\Omega$ one has $M_n(p)\to0$, so $u(p)\le0$, while $M_n\ge\beta_n$ gives $u\ge\kappa\log c$ on $\Omega$.
\medskip

It remains to prove that $u$ is nonconstant. Fix $0<r_0<1/4$ and, for sufficiently large $N$, set
 $p_N:=F(N)^{-1}(r_0,0)$. Since $(r_0,0)$ lies in the fixed attracting ball used in \Cref{prop:basic}, one has $p_N\in\Omega$. Starting from $(r_0,0)$, all subsequent coordinates are nonnegative real numbers and the first coordinate satisfies $z_{N+j}\ge r_0^{2^j}$ for all $j\geq0$. Consequently
$$
 u(p_N)\ge2^{-N}\log r_0.
$$
For sufficiently large $N$, this is strictly greater than the fixed negative number $u(0)=\kappa\log c$. Thus $u$ is a nonconstant bounded-above plurisubharmonic function on $\Omega$. The remaining defining properties of a Short $\C^2$ follow from \Cref{prop:basic}.
\end{proof}

\section{Two normalizations and contraction estimates}\label{sec:normalization}
We continue to work with the quadratic H\'enon family. In this section we establish several preliminary results used below. Set
$$
 Q_n:=\sum_{j=1}^nq_j,
 \qquad
 P_n:=\prod_{j=1}^na_j=c^{Q_n},
 \qquad
 P_0:=1,
$$
and
$$
 \lambda_n:=\frac{P_n}{a_{n+1}},
 \qquad n\ge0.
$$
Recall that $d_n=2q_n-q_{n+1}$ and observe that $\frac{\lambda_n}{\lambda_{n-1}}=c^{d_n}$.
Let $S(z,w)=(w,z)$ and define
\begin{equation}\label{eq:normalization}
 \Phi_n:=(DF(n)(0))^{-1}\circ F(n).
\end{equation}
Observe that $DF(n)(0)=P_nS^n$ and hence  $F(n)=P_nS^n\Phi_n$. Moreover, $\Phi_{n+1}=\sigma_n\circ\Phi_n$ where
$$
 \sigma_n=
 \begin{cases}
 (z,w+\lambda_nz^2),&n\text{ even},\\
 (z+\lambda_nw^2,w),&n\text{ odd}.
 \end{cases}
$$

These shears give a lower bound for the image of a fixed ball.

\begin{lemma}\label{lem:inner-radius}
Fix $0<r<\frac{1}{4}$ and set
$$
 \Sigma_n:=\sum_{j=0}^{n-1}\lambda_j, \qquad
 \beta_n(r):=\frac{P_n}{r^{-1}+\Sigma_n}.
$$
Then $ \B_{\beta_n(r)}\subset F(n)(\B_r)$. Moreover, if $\mu:=\log(1/c)$, then for all sufficiently large $n$,
$$
 \beta_n(r)\ge
 \frac{e^{-2\mu Q_n}}{n+r^{-1}}.
$$
\end{lemma}

\begin{proof}
Let $y\in\B_{\beta_n(r)}$ and put $p=F(n)^{-1}(y)$. It therefore suffices to show that $p\in\B_r$.

For $0\leq j\leq n$ set $u_j:=\|\Phi_j(p)\|$. Since $\Phi_j=\sigma_j^{-1}\circ\Phi_{j+1}$,
we obtain
$$
u_j\leq u_{j+1}+\lambda_ju_{j+1}^2.
$$
Consequently
$$
\frac1{u_j}
\geq
\frac1{u_{j+1}+\lambda_ju_{j+1}^2}
=
\frac1{u_{j+1}}
-\frac{\lambda_j}{1+\lambda_ju_{j+1}}
\geq
\frac1{u_{j+1}}-\lambda_j.
$$
Hence 
$$
\frac1{u_0}
\geq
\frac1{u_n}-\sum_{j=0}^{n-1}\lambda_j.
$$
Next observe that
$$\|y\|=\|F(n)(p)\|=P_n\|\Phi_n(p)\|=P_n u_n.$$
Therefore
$$
\frac1{u_0}
\geq
\frac{P_n}{\|y\|}-\Sigma_n.
$$
Using $\|y\|<\beta_n(r)$, this inequality implies that $\|p\|=u_0<r$. Since this holds for any $y\in\B_{\beta_n(r)}$ we get $\B_{\beta_n(r)}\subset F(n)(\B_r)$.

\medskip

For the second assertion, write $\mu=\log(1/c)$. Since $P_n=e^{-\mu Q_n}$
and, for $0\leq j<n$,
$$
\lambda_j=c^{Q_j-q_{j+1}}=e^{\mu(q_{j+1}-Q_j)}\leq e^{\mu Q_n},
$$
we have $\Sigma_n\leq ne^{\mu Q_n}$. Therefore
$$
\beta_n(r)
\geq
\frac{e^{-\mu Q_n}}
     {r^{-1}+ne^{\mu Q_n}}
=
\frac{e^{-2\mu Q_n}}
     {n+r^{-1}e^{-\mu Q_n}}
\geq
\frac{e^{-2\mu Q_n}}{n+r^{-1}}.
$$
\end{proof}

The next lemma {is} a version of the argument in \cite[Lemma~4.2]{FP}.

\begin{lemma}\label{lem:reset}
Assume $\liminf\rho_n=0$. For every compact $K\Subset\Omega$ there exist arbitrarily large $n$ such that
$$
 \sup_{p\in K}\norm{F(n)(p)}\le a_{n+1}.
$$
\end{lemma}

\begin{proof}
Put $\delta_n:=\sup_K\norm{F(n)}$. By \Cref{prop:basic}, $\delta_n\to0$. Choose $N$ arbitrarily large with $2\delta_N<c$. Since $\liminf\rho_n=0$, choose $m>N+1$ with $\rho_m\le2^{-N-1}$. Write $m=N+s+1$. Then $q_m\le2^s$.
\medskip

Suppose that $\delta_n>a_{n+1}$ for every $N\le n<m$. Then
$$
 \delta_{n+1}\le\delta_n^2+a_{n+1}\delta_n\le2\delta_n^2.
$$
Therefore
$$
 \delta_{N+s}\le(2\delta_N)^{2^s}<c^{2^s}.
$$
Since $N+s=m-1$, this inequality, together with the initial assumption, gives
$$
 c^{q_m}=a_m<\delta_{m-1}<c^{2^s}.
$$
Thus $q_m>2^s$, a contradiction.
\end{proof}

\begin{remark}\label{rem:rescale}
After discarding finitely many terms, we may assume that
$a_n\leq 1/4$ for every $n$. Conjugating the maps $F_n$ by a dilation gives
$$
\widetilde F_n(z,w)
:=4F_n(z/4,w/4)
=\left(\frac14 z^2+a_nw,a_nz\right),
$$
and hence
$$
\|\widetilde F_n(p)\|\leq \frac12\|p\|,
\qquad p\in\B.
$$
The estimates in \cref{lem:inner-radius} remain valid for
this rescaled family. Indeed, its normalized shears have
coefficients $\lambda_n/4$, so the same computation gives
$$
\B_{\widetilde\beta_n(r)}
\subset \widetilde F(n)(\B_r),
\qquad
\widetilde\beta_n(r)
:=\frac{P_n}{r^{-1}+\frac14\Sigma_n}
\geq \frac{e^{-2\mu Q_n}}{n+r^{-1}},
\qquad 0<r<1.
$$
Moreover, \cref{lem:reset} remains valid, since its proof
uses the estimate
$$
\delta_{n+1}
\leq \frac14\delta_n^2+a_{n+1}\delta_n
\leq \delta_n^2+a_{n+1}\delta_n.
$$
\end{remark}
\medskip

There is another natural way to normalize the compositions $F(n)$. Rather
than removing the accumulated linear part, we rescale the orbit at time
$n$ by the coefficient that appears in the next map. More precisely,
let
$$
\tau_n(z,w):=\frac1{a_{n+1}}(z,w).
$$
The associated non-autonomous change of coordinates transforms
$F_n$ into
$$
\tau_n\circ F_n\circ \tau_{n-1}^{-1}=\frac{a_n^2}{a_{n+1}}Q=c^{d_n}Q,
$$
where $Q(z,w):=(z^2+w,z)$. This motivates the normalization
$$J_n:=\tau_n\circ F(n)=\frac{F(n)}{a_{n+1}}=\lambda_nS^n\Phi_n.$$
It satisfies the exact recurrence
$$J_{n+1}=c^{d_{n+1}}Q(J_n).$$
Moreover,
$$
\|J_n(p)\|\leq1
\quad\Longleftrightarrow\quad
\|F(n)(p)\|\leq a_{n+1},
$$
so \Cref{lem:reset} provides indices at which this
normalization is uniformly bounded on a prescribed compact subset of
$\Omega$.

\begin{proposition}\label{prop:calibrated-collapse}
Assume $\liminf\rho_n=0$. If $(K_j)$ is a compact exhaustion of $\Omega$, then there exists an increasing sequence $n_j$ such that
$$
 J_{n_j}\longrightarrow0
$$
locally uniformly on $\Omega$.
\end{proposition}

\begin{proof}
By \Cref{lem:reset}, choose $n_j$ increasing so that
$$
 \sup_{K_j}\norm{J_{n_j}}\le1.
$$
Thus $(J_{n_j})$ is locally bounded on $\Omega$. Every locally uniform subsequential limit has bounded holomorphic coordinate functions. Since the infinitesimal Kobayashi metric of $\Omega$ vanishes identically, every bounded holomorphic function on $\Omega$ is constant. All $J_n$ fix the origin, so every subsequential limit is identically zero. Hence the full sequence $J_{n_j}$ converges locally uniformly to zero.
\end{proof}

\section{A criterion for parabolicity in the quadratic H\'enon family}\label{sec:liouville}
We now use the preceding results to prove Theorem \ref{thm:second-classification}.

\begin{proof}[Proof of Theorem \ref{thm:second-classification}] Without loss of generality, we may discard finitely many maps $F_n$ so that all remaining maps satisfy $F_n(\B_{1/4})\subset\B_{1/4}$. It suffices to prove the theorem for their conjugates
$$\widetilde F_n(z,w):=4F_n(z/4,w/4)=\left(\frac14 z^2+a_nw,a_nz\right)$$ 
which satisfy $\widetilde F_n(\B)\subset \B$. Fix $0<r<1$, and let $\widetilde\beta_n(r)$ be as in Lemma \ref{lem:inner-radius} and \cref{rem:rescale}, so that $\B_{\widetilde\beta_n(r)}\subset \widetilde F(n)(\B_r)$ and
$$
\widetilde\beta_n(r)
\geq
\frac{e^{-2\mu Q_n}}{n+r^{-1}},
\qquad
\mu:=\log\frac1c.
$$

Since $a_n=c^{q_n}$, the assumption \eqref{eq:henon-parabolic} is equivalent, after shifting the index, to
\begin{equation}\label{eq:Q_n estimate}
    Q_n\leq mq_{n+1}
\end{equation}
for all sufficiently large $n$. Choose an integer $\ell>2m$. We claim that
$$
a_{n+1}^{\ell}<\widetilde\beta_n(r)
$$
for all sufficiently large $n$. Indeed,
$$
a_{n+1}^{\ell}
=
e^{-\ell\mu q_{n+1}},
$$
whereas
$$
\widetilde\beta_n(r)
\geq
\frac{e^{-2\mu Q_n}}{n+r^{-1}}
\geq
\frac{e^{-2m\mu q_{n+1}}}{n+r^{-1}}.
$$
Thus it is enough to have
$$
\log(n+r^{-1})<(\ell-2m)\mu q_{n+1},
$$
{which holds for all sufficiently large $n$. Indeed, $q_j\to\infty$ implies $Q_n/n\to\infty$, and \eqref{eq:Q_n estimate} gives $q_{n+1}\ge Q_n/m$, so $\log n=o(q_{n+1})$.}
\medskip

Now fix $p\in\Omega$. By Lemma \ref{lem:reset} there exist arbitrarily large $n$ such that
$\|\widetilde F(n)(p)\|\leq a_{n+1}$.
For all such sufficiently large indices $a_{n+1}^{\ell}<\widetilde\beta_n(r)$,
and hence
$$
\|\widetilde F(n)(p)\|
\leq
a_{n+1}
<
\widetilde\beta_n(r)^{1/\ell}.
$$

Therefore, the hypotheses of Lemma \ref{lem:extremal-visibility} are
satisfied with $\eta:=\frac1\ell$.
Together with Lemma \ref{lem:extremal-exhaustion}, this gives $u^*_{\overline{\B}_r,\Omega}
\leq
-\frac1\ell
$
on $\Omega$, and Lemma \ref{lem:negative-extremal-parabolic} finally implies
that every bounded-above plurisubharmonic function on $\Omega$ is
constant.
\end{proof}

\begin{lemma}\label{lem:Q-ratio}
If $q_{n+1}/q_n\to2$, then there exists an integer $m>1$ such that $Q_n\leq mq_{n+1}$ for all sufficiently large $n$.
\end{lemma}

\begin{proof}
Choose $1<s<2$ and $N$ such that $q_{j+1}\geq s q_j$ for all $j\geq N$. For $n\geq N$, the finitely many earlier terms must be kept separate:
$$
\frac{Q_n}{q_{n+1}}
\leq\frac{Q_{N-1}}{q_{n+1}}+\sum_{\ell=1}^{n-N+1}s^{-\ell}
\leq\frac{Q_{N-1}}{q_{n+1}}+\frac1{s-1}.
$$
Since $q_{n+1}\to\infty$, an integer $m>1/(s-1)$ satisfies $Q_n\leq m q_{n+1}$ for all sufficiently large $n$.
\end{proof}

\begin{corollary}\label{cor:regular-short}
Assume $q_{n+1}/q_n\to2$. Then
$$
 \Omega\text{ is a Short }\C^2
 \quad\Longleftrightarrow\quad
 \liminf_{n\to\infty}\rho_n>0.
$$
\end{corollary}

\begin{proof}
If $\liminf\rho_n>0$, apply \Cref{thm:short}. If $\liminf\rho_n=0$, then \Cref{lem:Q-ratio} gives $Q_n\le m q_{n+1}$, which is equivalent to the product hypothesis in \eqref{eq:henon-parabolic}. Hence every bounded-above plurisubharmonic function on $\Omega$ is constant, so $\Omega$ cannot be a Short $\C^2$.
\end{proof}

\section{Derivative normalization in the \texorpdfstring{$\lambda$}{lambda}-summable case}\label{sec:summable}
Recall that the map $\Phi_N$ defined in \eqref{eq:normalization} is an alternating composition of two types of shears $\sigma_n$. Therefore we have
$$
 \Phi_N(z,w)=\operatorname{id}+
 \left(
 \sum_{\substack{0\le n<N\\ n\ \mathrm{odd}}}\lambda_nw^2,
 \sum_{\substack{0\le n<N\\ n\ \mathrm{even}}}\lambda_nz^2
 \right){+O(\|(z,w)\|^3).}
$$
{Because all $\lambda_n$ are positive, convergence of $\Phi_N$ on a neighborhood of the origin forces convergence of both quadratic coefficient sums, equivalently} 
\begin{equation}\label{eq:summability-intro}
 \sum_{n=0}^{\infty}\lambda_n<\infty.
\end{equation}
Throughout this section we assume that the sum in \eqref{eq:summability-intro} is finite. The following theorem shows that this condition is also sufficient for the convergence of $(\Phi_n)$ on some neighborhood of the origin.

\begin{theorem}\label{thm:summable-core}
The maps $\Phi_n^{-1}$ converge locally uniformly on $\C^2$ to an injective entire map 
$\Psi:\C^2\longrightarrow\C^2$. Its image $\Ucal:=\Psi(\C^2)$ is a Fatou--Bieberbach domain contained in $\Omega$, and
$$
 \Phi_n\longrightarrow\Psi^{-1}
$$
locally uniformly on $\Ucal$.
\end{theorem}

\begin{proof}
Write $\Psi_n:=\Phi_n^{-1}$. Fix $R>0$ and choose $N$ so large that
$$
 \sum_{n=N}^{\infty}\lambda_n<\frac1{4R}.
$$
For $m>N$ set
$$
 \mathcal R_{N,m}:=\sigma_N^{-1}\circ\cdots\circ\sigma_{m-1}^{-1},
$$
so $\Psi_m=\Psi_N\circ\mathcal R_{N,m}$. If $p_m=p\in\B_{2R}$ and $p_j=\sigma_j^{-1}(p_{j+1})$, then, with $u_j=\norm{p_j}$,
$$
 u_j\le u_{j+1}+\lambda_ju_{j+1}^2,
 \qquad
 \frac1{u_j}\ge\frac1{u_{j+1}}-\lambda_j.
$$
Hence 
$$\frac1{u_j}\ge\frac1{u_{m}}-\sum_{k=j}^{m-1}\lambda_k>\frac{1}{2R}-\frac{1}{4R}=\frac{1}{4R}.$$ 
Therefore $u_j<4R$ for $N\le j\le m$ and 
$$
 \sup_{\B_{4R}}\norm{D\sigma_j^{-1}-I}\le8R\lambda_j,
$$
so the maps $\mathcal R_{N,m}$ have uniformly bounded Lipschitz constants on $\B_{2R}$. Indeed,
 by the chain rule,
$$
\sup_{\mathbb B_{2R}}\|D\mathcal R_{N,m}\|
\leq
\prod_{j=N}^{m-1}(1+8R\lambda_j)
\leq
\exp\left(8R\sum_{j=N}^{\infty}\lambda_j\right)
=:C_R<\infty.
$$
Moreover,
$$
 \sup_{\B_R}\norm{\mathcal R_{N,m+1}-\mathcal R_{N,m}}
 \le C_RR^2\lambda_m.
$$
Hence $(\mathcal R_{N,m})$ is uniformly Cauchy on $\B_R$, and therefore $(\Psi_m)$ converges locally uniformly on $\C^2$. The limit $\Psi$ is nondegenerate because $D\Psi_n(0)=I$. By Hurwitz's theorem, $\Psi$ is injective, and we set $\Ucal=\Psi(\C^2)$. The convergence of the inverses gives $\Phi_n\to\Psi^{-1}$ locally uniformly on $\Ucal$.

If $p\in\Ucal$, then $(\Phi_n(p))$ is bounded. Since $F(n)(p)=P_nS^n\Phi_n(p)$
and $P_n\to0$, we obtain $F(n)(p)\to0$. Thus $\Ucal\subset\Omega$.
\end{proof}

\begin{proposition}\label{prop:bounded-core}
One has
$$
 \Ucal=\{p\in\Omega:\sup_n\norm{\Phi_n(p)}<\infty\}.
$$
\end{proposition}

\begin{proof}
The inclusion from left to right follows from \Cref{thm:summable-core}. Conversely, suppose $p\in\Omega$ and $(\Phi_n(p))$ is bounded. Choose a subsequence with $\Phi_{n_j}(p)\to x$. Since $\Phi_{n_j}^{-1}\to\Psi$ locally uniformly, we get

\begin{align*}
\|p-\Psi(x)\|
&=
\|\Phi_{n_j}^{-1}(\Phi_{n_j}(p))-\Psi(x)\|\\
&\leq
\sup_{\mathbb B_R}\|\Phi_{n_j}^{-1}-\Psi\|
+
\|\Psi(\Phi_{n_j}(p))-\Psi(x)\|
\longrightarrow0.
\end{align*}
Therefore $p=\Psi(x)\in\Ucal$.
\end{proof}

We have shown that $(\Phi_n(p))$ is bounded if and only if $p\in\Ucal$. We next establish lower bounds for $\|\Phi_n(p)\|$ when $p\in\Omega\setminus\Ucal$.

\begin{lemma}\label{lem:outside-core} Let $ L_n:=\sum_{j=n}^{\infty}\lambda_j.$
If $p\in\Omega$ and $L_n\|\Phi_n(p)\|<1$ for some $n$, then $p\in\Ucal$. Consequently,
$$
p\in\Omega\setminus\Ucal
\quad\Longrightarrow\quad
\|\Phi_n(p)\|\ge\frac1{L_n}
$$
for every sufficiently large $n$.
\end{lemma}

\begin{proof}
Set $u_j:=\|\Phi_j(p)\|$.
If $u_n=0$, then $\Phi_n(p)=0$, hence $p=0\in\Ucal$. Thus we may assume
$u_n>0$. Since $u_{j+1}\leq u_j+\lambda_j u_j^2$, it follows that
$$
\frac1{u_{j+1}}
\geq
\frac1{u_j+\lambda_j u_j^2}
\geq
\frac1{u_j}-\lambda_j.
$$
Hence
$$
\frac1{u_j}
\geq
\frac1{u_n}-
\sum_{\ell=n}^{j-1}\lambda_\ell
\geq
\frac1{u_n}-L_n.
$$
By assumption, $L_nu_n<1$, so the right-hand side is positive. Hence
$$
u_j
\leq
\frac{u_n}{1-L_nu_n}
$$
for all $j\geq n$. Thus $(\Phi_j(p))$ is bounded, and
\Cref{prop:bounded-core} implies that $p\in\Ucal$.
\end{proof}

\subsection{Capture of the full basin}\label{sec:capture} The aim of this subsection is to prove Theorem \ref{thm:psh-capture}, which is a corollary of the more general Theorem \ref{thm:alpha-capture}. Throughout this subsection we assume that 
$$\liminf_{n\to\infty}\rho_n=0.$$
Under the summability assumption \eqref{eq:summability-intro}, \Cref{thm:summable-core} provides a Fatou--Bieberbach domain $\Ucal\subset\Omega$. We now compare $\Ucal$ with the full basin $\Omega$. 

Recall the notation introduced in the introduction
$$
 \alpha_n:=\frac{\log(1/L_n)}{\log(1/\lambda_n)}.
$$
Since for all sufficiently large $n$ we have $0<\lambda_n<L_n<1$, it follows that $0<\alpha_n<1$. We first construct the plurisubharmonic function used in the capture argument.

\begin{lemma}\label{lem:psh-separator}
{Let $(K_j)$ be a compact exhaustion of $\Omega$ and let $(n_j)$ be an increasing sequence satisfying $\sup_{K_j}\|J_{n_j}\|\le1$, as constructed in the proof of \Cref{prop:calibrated-collapse}.} Then
$$
 G:=\left(\limsup_{j\to\infty}
 \frac{\log^+\norm{\Phi_{n_j}}}{\log(1/\lambda_{n_j})}\right)^*
$$
is plurisubharmonic on $\Omega$, satisfies $0\le G\le1$, and in particular $G|_\Ucal=0$. If $p\in\Omega\setminus\Ucal$, then
$$
 G(p)\ge\limsup_{j\to\infty}\alpha_{n_j}.
$$
\end{lemma}

\begin{proof}
Set
$$g_j:=\frac{\log^+\norm{\Phi_{n_j}}}{\log(1/\lambda_{n_j})}.$$
On $K_j$, the bound $\norm{J_{n_j}}\le1$ gives  $\norm{\Phi_{n_j}}\le\frac1{\lambda_{n_j}}$,
so the functions $g_j$ are locally uniformly bounded between $0$ and $1$. They converge locally uniformly to zero on $\Ucal$ by \Cref{thm:summable-core}. Hence $G$ is plurisubharmonic, takes values in $[0,1]$, and vanishes on $\Ucal$.

If $p\notin\Ucal$, then \Cref{lem:outside-core} gives $\norm{\Phi_{n_j}(p)}\ge\frac1{L_{n_j}}$ for all sufficiently large $j$ and therefore
$$
g_j(p)\geq\alpha_{n_j}.
$$
for all sufficiently large $j$, proving the assertion.
\end{proof}

If, in addition, $G>\delta>0$ on $\Omega\setminus\Ucal$, then the following lemma shows that $\Omega=\Ucal$.

\begin{lemma}\label{lem:psh-gap}
Let $D\subset\C^k$ be connected and let $u$ be bounded above and
plurisubharmonic on $D$. Suppose that $V\subset D$ is a nonempty open
set and that there are real numbers $a<b$ such that
$$
u=a\quad\text{on }V,
\qquad
u\geq b\quad\text{on }D\setminus V.
$$
Then $V=D$.
\end{lemma}

\begin{proof}

Suppose that $V\neq D$. Since $D$ is connected, $bV\cap D\neq\varnothing$. Choose $q\in V$ sufficiently close to $bV$ that
$$
R:=\operatorname{dist}(q,b V)<\operatorname{dist}(q,b D).
$$
Choose $p\in b V$ with $\|p-q\|=R$. Then $\B(q,R)\subset V$, $\overline{\B}(q,R)\subset D$, and $p\in b\B(q,R)$.

Let
$$
\vartheta_r=\frac{\operatorname{vol}(\B(p,r)\cap \B(q,R))}
     {\operatorname{vol}(\B(p,r))}
$$
Observe that $\vartheta_r\rightarrow\frac1{2}$ as $r\rightarrow 0$. Since $u$ is upper semicontinuous, for every
$\varepsilon>0$ we may choose $r>0$ sufficiently small so that
$$
u\leq u(p)+\varepsilon
\qquad\text{on }\B(p,r).
$$
On $\B(p,r)\cap \B(q,R)$ we have $u=a$. Hence the sub-mean value property gives
$$
u(p)
\leq
\vartheta_r a+(1-\vartheta_r)(u(p)+\varepsilon).
$$
Therefore
$$
u(p)-a
\leq
\frac{1-\vartheta_r}{\vartheta_r}\varepsilon.
$$
Letting $\varepsilon\downarrow0$ yields $u(p)\leq a$. Finally, since $p\in D\setminus V$ we have $u(p)\geq b>a$, which is a contradiction. 
\end{proof}

We are now ready to prove \Cref{thm:psh-capture}.

\begin{proof}[Proof of \Cref{thm:psh-capture}]
Put $\delta=\liminf_{n\to\infty}\alpha_n>0$. \Cref{lem:psh-separator} gives $G\in\PSH(\Omega)$ with
$$
0\le G\le1,\qquad G|_{\Ucal}=0,\qquad
G\ge\delta\quad\text{on }\Omega\setminus\Ucal.
$$
The separation lemma, \Cref{lem:psh-gap}, implies $\Omega=\Ucal$.
\end{proof}

\medskip

Next we observe that \Cref{lem:psh-separator} suggests that the condition in \Cref{thm:psh-capture} can be weakened further. Define
$$
 T(N):=\min\{m>N:\rho_m\le2^{-N-1}\}.
$$
Then the proof of \Cref{lem:reset} shows that whenever $2\sup_K\norm{F(N)}<c$, there is some $n\in[N,T(N))$ with $ \sup_K\norm{F(n)}\le a_{n+1}$.
Finally we define
$$
 \alpha_N^*:=\min_{N\le n<T(N)}\alpha_n,
 \qquad
 \alpha_*:=\limsup_{N\to\infty}\alpha_N^*.
$$
The following result improves  \Cref{thm:psh-capture}.

\begin{theorem}\label{thm:alpha-capture}
If $\alpha_*>0$, then $\Omega=\Ucal\simeq\C^2$.
\end{theorem}

\begin{proof}
Let $K_1\Subset K_2\Subset\cdots\Subset\Omega$ be a compact exhaustion. Choose an increasing sequence $(N_j)$ such that $\alpha_{N_j}^*\to\alpha_*$. By passing to a subsequence we may assume
$$
 N_j>T(N_{j-1}),
 \qquad
 2\sup_{K_j}\norm{F(N_j)}<c.
$$
For each $j$, the reset-window statement gives $n_j\in[N_j,T(N_j))$ such that
$$
 \sup_{K_j}\norm{J_{n_j}}\le1.
$$
The sequence $(n_j)$ is increasing and  $\alpha_{n_j}\ge\alpha_{N_j}^*$.
Let $G$ be the function from \Cref{lem:psh-separator}. If $\Omega\setminus\Ucal$ were nonempty, then 
$$
 G|_{\Ucal}=0,
 \qquad
 G|_{\Omega\setminus\Ucal}\ge\frac{\alpha_*}{2}>0.
$$
This contradicts \Cref{lem:psh-gap} and therefore $\Omega=\Ucal$.
\end{proof}

Finally, we show that these results imply part $(ii)$ of Theorem \ref{thm:first-classification}.

\begin{proposition}\label{prop:positive-defect}
If
$$
 \liminf_{n\to\infty}d_n>0,
$$
then $\sum_n\lambda_n<\infty$ and $\alpha_n\to1$.
\end{proposition}

\begin{proof}
Choose $D>0$ such that $d_n\ge D$ for all sufficiently large $n$ and observe that 
$$
 \frac{\lambda_n}{\lambda_{n-1}}=c^{d_n}\le c^D=:r<1.
$$
 Hence $\lambda_n$ decays geometrically and is summable. Moreover
$$
 \lambda_n\le L_n\le\frac{\lambda_n}{1-r}
$$
for all sufficiently large $n$. Therefore
$$
 0\le\log\frac{L_n}{\lambda_n}\le\log\frac1{1-r},
$$
while $\log(1/\lambda_n)\to\infty$. It follows that
$$
 \alpha_n
 =1-\frac{\log(L_n/\lambda_n)}{\log(1/\lambda_n)}
 \longrightarrow1.
$$
\end{proof}

\begin{corollary} If $\liminf\rho_n=0$ and $\liminf d_n>0$, then $\Omega\simeq\C^2$.
\end{corollary}

\begin{proof}
By \Cref{prop:positive-defect} we have $\alpha_n\to1$, so we can apply \Cref{thm:psh-capture}.
\end{proof}

\section{Failure of derivative-normalized capture}\label{sec:proper-core}

{When $\alpha_*>0$, \Cref{thm:alpha-capture} gives $\Omega=\Ucal$. We show that equality can fail when $\alpha_*=0$, even if $q_{n+1}/q_n\to2$. By \Cref{lem:outside-core}, $\|\Phi_n(p)\|\geq L_n^{-1}$ for $p\in\Omega\setminus\Ucal$. This motivates the rescaling} 
$$K_n:=L_n\Phi_n,$$
{for which } 
\begin{equation}\label{eq:core}
\{p\in\Omega:K_n(p)\to0\}=\Ucal.
\end{equation}
Next we define
$$
G_n(z,w)=
\begin{cases}
((1-r_n)z,(1-r_n)(w+r_nz^2)),&n\text{ even},\\
((1-r_n)(z+r_nw^2),(1-r_n)w),&n\text{ odd}.
\end{cases}
$$
{where $r_n:=\lambda_n/L_n$ satisfies $0<r_n<1$. Thus} 
$$K_{n+1}=G_n\circ K_n.$$

When $r_n$ is small, the maps $G_n$ are close to the identity. Pairing two consecutive maps $G_n$ reveals the vector field governing their dynamics.
Put
$$
h_m:=r_{2m}+r_{2m+1},
\qquad
\Gamma_m:=G_{2m+1}\circ G_{2m}.
$$ 
and observe that
\begin{equation}\label{eq:pair-flow-expansion}
\Gamma_m=\operatorname{id}+h_mV+R_m,
\qquad
V(z,w)=\left(-z+\frac12w^2,-w+\frac12z^2\right),
\end{equation}

For every compact $K\subset\C^2$, direct expansion of the polynomial maps gives the supremum-norm estimate
\begin{equation}\label{eq:pair-error}
\sup_{x\in K}\norm{R_m(x)}
\le C_K\bigl(h_m^2+\abs{r_{2m+1}-r_{2m}}\bigr).
\end{equation}

The vector field $V$ has three nonzero {equilibrium points}
$$
p_\omega=(2\omega,2\omega^2),\qquad \omega^3=1,
$$
and $\operatorname{spec}DV(p_\omega)=\{1,-3\}$.

\subsection{Stable disks for the rescaled dynamics}

{We prove the following version of the non-autonomous stable-manifold argument. }

%%%%%%%%%%%%%%%%%%%%%%%%%%%%%%%%%%%%%%%%%%
\begin{lemma}\label{lem:asymptotic-flow-stable} Let $p_\omega$ be the {equilibrium point} of the vector field $V$ and let $D$ be {a} neighborhood of $p_\omega$, on which \eqref{eq:pair-error} holds. If
$$
h_m>0,\qquad h_m\to0,\qquad \sum_m h_m=\infty,
$$
and, on every relatively compact neighborhood $D'\Subset D$ of $p_\omega$,
$$
\sum_m\left(h_m^2+{\sup_{x\in D'}\norm{R_m(x)}}\right)<\infty.
$$
Then, after discarding finitely many terms, there exists a local holomorphic disk $ W_\omega$ such that
$$
\Gamma_{m-1}\circ\cdots\circ\Gamma_M(x)\longrightarrow p_\omega
$$
locally uniformly for $x\in W_\omega$.
\end{lemma}

\begin{proof}
{\bf Step 1.:} {\it We first prove that finite compositions of sufficiently late blocks of maps $\Gamma_j$  are arbitrary small perturbations of a flow of vector field $V$.}
\medskip

Choose a compact neighborhood $D'\Subset D$ of $p_\omega$. 
Let $\varphi^t$ denote the local flow of $V$. Fix $\tau>0$ small. Since
$$
\norm{\varphi^h-\operatorname{id}-hV}_{\infty}\le Ch^2,
$$
the assumptions imply
$$
\norm{\Gamma_m-\varphi^{h_m}}_{\infty}
\le
C\left(h_m^2+\norm{R_m}_{\infty}\right)
$$
{Here we are taking supremum norms on the fixed neighborhood $D'$.}

Choose a sequence $M=M_0<M_1<\cdots$ so that $M_{j+1}$ is the first index satisfying
$$
\tau_j:=\sum_{m=M_j}^{M_{j+1}-1}h_m\ge\tau.
$$
Then observe that $\tau_j\to\tau$. Next we define
$$
B_j:=\Gamma_{M_{j+1}-1}\circ\cdots\circ\Gamma_{M_j},
$$
and observe that the usual group property of the flow together with the summability of the  errors, gives uniform convergence
\begin{equation}\label{eq:block-convergence}
B_j\longrightarrow B:=\varphi^\tau
\end{equation}
on a fixed neighborhood of $p_\omega$. This completes step 1.

\medskip 
Recall that 
$\operatorname{spec}DV(p_\omega)=\{-3,1\}$, hence after translating $p_\omega$ to the origin and taking eigen-coordinates $(s,u)$,
$$
DB(0)
=
\begin{pmatrix}
e^{-3\tau}&0\\
0&e^\tau
\end{pmatrix}.
$$
Set
$$
\lambda:=e^{-3\tau},
\qquad
a:=e^{-\tau}.
$$
On a sufficiently small bidisk $|s|,|u|<R$ write
$$
B_j(s,u)
=
\left(
\lambda s+f_j(s,u),
a^{-1}u+g_j(s,u)
\right).
$$
{The maps $B_j$ are holomorphic, so the local uniform convergence in \eqref{eq:block-convergence} also gives convergence of their first derivatives on a smaller bidisk by Cauchy's estimates.}
By \eqref{eq:block-convergence}, after shrinking $R$ and discarding finitely many terms,
$$
\norm{Df_j},\norm{Dg_j}\le\epsilon,
\qquad
\delta_j:=\max\{|f_j(0,0)|,|g_j(0,0)|\}\longrightarrow0,
$$
where $\epsilon>0$ can be chosen as small as needed.
\medskip

{\bf Step 2.:} {\it We prove the existence of stable graphs for the sequence $B_j$, that converges to the locally stable manifold of $\varphi^\tau$.}
\medskip

Fix $0<\kappa<1$, and let $\mathcal H$ be the space of holomorphic functions $\gamma:\Delta_R\to\C$ satisfying
$$
\norm{\gamma}_\infty\le R/2,
\qquad
\operatorname{Lip}(\gamma)\le\kappa.
$$
Equip $\mathcal H^{\mathbb N}$ with the metric
$$
d(\gamma,\widetilde\gamma)
:=
\sup_j\norm{\gamma_j-\widetilde\gamma_j}_\infty,
$$
and note that this is a complete metric space.
\medskip 

For $\gamma=(\gamma_j)\in\mathcal H^{\mathbb N}$ define a self map $\mathcal G: {\mathcal H^{\mathbb N}\to\mathcal H^{\mathbb N}}$ by
\begin{equation}\label{eq:explicit-graph-transform}
(\mathcal G\gamma)_j(s)
:=
a\left[
\gamma_{j+1}\!\left(
\lambda s+f_j(s,\gamma_j(s))
\right)
-
g_j(s,\gamma_j(s))
\right].
\end{equation}
After discarding further terms we may assume $\delta:=\sup_j\delta_j$ is arbitrarily small. For $|s|<R$,
$$
\left|
\lambda s+f_j(s,\gamma_j(s))
\right|
\le
\left(
\lambda+\frac{\delta}{R}+\frac32\epsilon
\right)R.
$$
Thus, for $\epsilon$ and $\delta/R$ sufficiently small, the argument of $\gamma_{j+1}$ remains in $\Delta_R$. The same choices give
$$
\norm{(\mathcal G\gamma)_j}_\infty\le R/2.
$$
Moreover,
$$
\operatorname{Lip}((\mathcal G\gamma)_j)
\le
a\left[
\kappa\lambda+\epsilon(1+\kappa)^2
\right].
$$
Since $a\lambda=e^{-4\tau}<1$, we may choose $\epsilon$ so small that the last quantity is at most $\kappa$. Hence
$$
\mathcal G(\mathcal H^{\mathbb N})
\subset
\mathcal H^{\mathbb N}.
$$

If $\gamma,\widetilde\gamma\in\mathcal H^{\mathbb N}$ then, using \eqref{eq:explicit-graph-transform},
$$
\norm{(\mathcal G\gamma)_j-(\mathcal G\widetilde\gamma)_j}_\infty
\le
a\left[1+\epsilon(1+\kappa)\right]d(\gamma,\widetilde\gamma).
$$
Since $a<1$, after decreasing $\epsilon$ we obtain a constant $q<1$ such that
$$
d(\mathcal G\gamma,\mathcal G\widetilde\gamma)
\le q\,d(\gamma,\widetilde\gamma).
$$
Thus $\mathcal G$ is a contraction, and {let} $(\gamma_j)$ be its unique fixed point.
\medskip

The fixed-point identity is
$$
\gamma_j(s)
=
a\left[
\gamma_{j+1}\!\left(
\lambda s+f_j(s,\gamma_j(s))
\right)
-
g_j(s,\gamma_j(s))
\right],
$$
or equivalently
$$
a^{-1}\gamma_j(s)+g_j(s,\gamma_j(s))
=
\gamma_{j+1}\!\left(
\lambda s+f_j(s,\gamma_j(s))
\right).
$$
Hence the graphs
$$
W_j^s
:=
\{(s,\gamma_j(s)):|s|<R\}
$$
satisfy
$$
B_j(W_j^s)\subset W_{j+1}^s.
$$
To complete step 2. it remains to show that these graphs approach the stable graph of the limiting map. Let
$$
B(s,u)
=
\left(
\lambda s+f(s,u),
a^{-1}u+g(s,u)
\right)
$$
be the limit in \eqref{eq:block-convergence}, and let $\mathcal G_\infty$ be the operator obtained from \eqref{eq:explicit-graph-transform} by replacing $f_j,g_j$ by $f,g$. Its unique fixed point is a constant sequence $(\gamma_*,\gamma_*,\ldots)$. 
Since $B(0)=0$, the fixed-point equation at $s=0$, together with $\gamma_*$ being $\kappa$-Lipschitz, gives
$$
|\gamma_*(0)|\leq a[1+\epsilon(1+\kappa)]|\gamma_*(0)|.
$$
The coefficient is less than one by the choice made above, so $\gamma_*(0)=0$.

If for every $J$, we compare the fixed point $(\gamma_j)_{j\ge J}$ with the constant sequence $(\gamma_*)$, we get
$$
\sup_{j\ge J}\norm{\gamma_j-\gamma_*}_\infty
\le
\frac{1}{1-q}
\sup_{j\ge J}
\norm{
(\mathcal G\gamma_*)_j-\gamma_*
}_\infty.
$$
By \eqref{eq:block-convergence}, the right-hand side tends to zero as $J\to\infty$. Therefore $\gamma_j\longrightarrow\gamma_*$ uniformly and in particular $
\gamma_j(0)\longrightarrow0$. 

\medskip
{\bf Step 3.:} {\it We prove that on $W_0^s$ $\Gamma$-orbits converge to $p_\omega$.}
\medskip

{Let $(s_j,u_j)\in W_j^s$ be a block-time orbit, and let $x_j$ denote the corresponding point in the original $(z,w)$ coordinates.} Then $u_j=\gamma_j(s_j)$, and so 
$$
|u_j|
\le
|\gamma_j(0)|+\kappa|s_j|.$$
On the other hand $s_{j+1}
=
\lambda s_j+f_j(s_j,u_j)$,
and hence, for some $b<1$,
$$
|s_{j+1}|
\le
b|s_j|+\eta_j,
\qquad
\text{where }\eta_j\longrightarrow0.
$$
It follows that $s_j\to0$, locally uniformly on $W_0^s$, and therefore also $u_j\to0$. Thus the $B_j$-orbit converges to $p_\omega$.

It remains to control iterates within each block. For $M_j\le n<M_{j+1}$, set $$t_{j,n}:=\sum_{m=M_j}^{n-1}h_m.$$ The choice of $M_{j+1}$ gives $0\le t_{j,n}<\tau$. On a fixed neighborhood of $p_\omega$, the flow property and the summable errors yield
$$
\Gamma_{n-1}\circ\cdots\circ\Gamma_{M_j}(x_j)
=\varphi^{t_{j,n}}(x_j)+o(1),
$$
uniformly in $n$ and locally uniformly in the initial point. Since $x_j\to p_\omega$ and $\varphi^t(p_\omega)=p_\omega$, we obtain
$$
\begin{aligned}
\sup_{M_j\le n<M_{j+1}}
\norm{\Gamma_{n-1}\circ\cdots\circ\Gamma_{M_j}(x_j)-p_\omega}
&\le C\norm{x_j-p_\omega}
+C\sum_{m\ge M_j}\bigl(h_m^2+\sup_{x\in D_0}\norm{R_m(x)}\bigr)
\end{aligned}
$$
where each of the two terms on the right side converge to zero as $j\to\infty$.
Thus the full $\Gamma$-orbit converges to $p_\omega$.
Taking $
W_\omega:=W_0^s$ in the original coordinates completes the proof.
\end{proof}

{We can now obtain points in the basin that do not belong to $\Ucal$.}

\begin{theorem}\label{thm:stable-disks}
Assume
$$
\sum_n\lambda_n<\infty,
\qquad
\sum_n r_n^2<\infty,
\qquad
\sum_m\abs{r_{2m+1}-r_{2m}}<\infty.
$$
Then, for every $\omega^3=1$, there exists a holomorphic disk $\widetilde W_\omega\subset\Omega$ such that
$$
K_n(p)\longrightarrow p_\omega
$$
locally uniformly for $p\in \widetilde W_\omega$. Consequently, $\Ucal\subsetneq\Omega$.
\end{theorem}

\begin{proof}
The square-summability assumption gives $r_n\to0$. Moreover,
$$
\prod_{n=N}^{M}(1-r_n)=\frac{L_{M+1}}{L_N}\longrightarrow0.
$$
Since $r_n\to0$, this implies $\sum_nr_n=\infty$. Hence
$$
h_m=r_{2m}+r_{2m+1}\to0,
\qquad
\sum_mh_m=\infty.
$$
Furthermore, $h_m^2\le2(r_{2m}^2+r_{2m+1}^2)$. Together with our assumptions and \eqref{eq:pair-error} this implies {that, for every relatively compact neighborhood $D'\Subset D$,}
$$
\sum_m\left(h_m^2+{\sup_{x\in D'}\norm{R_m(x)}}\right)<\infty.
$$

By \Cref{lem:asymptotic-flow-stable} there exists an integer $M$ and a one-dimensional local holomorphic disk $W_\omega$  such that
$$
\Gamma_{m-1}\circ\cdots\circ\Gamma_M(x)\longrightarrow p_\omega
$$
locally uniformly for $x\in W_\omega$. {Define $\widetilde W_\omega:=K_{2M}^{-1}(W_\omega)$.}

Since $K_{2M}$ is an automorphism $\widetilde W_\omega$ is a holomorphic disk in the original coordinates on which the sequence $K_{2m}\to p_\omega$ locally uniformly.

Since $r_{2m}\to0$, the maps $G_{2m}$ converge uniformly to the identity on a neighborhood of $p_\omega$. Therefore
$$
K_{2m+1}=G_{2m}\circ K_{2m}\longrightarrow p_\omega
$$
locally uniformly on {$\widetilde W_\omega$}, and hence $K_n\to p_\omega$ locally uniformly there.

\medskip

In particular $(K_n)$ is locally bounded on $\widetilde W_\omega$ and does not converge to zero. Since
$\|F(n)\|=a_{n+1}r_n\|K_n\|$
it clearly follows that $F(n)|_{ \widetilde W_\omega}\rightarrow 0$, because $a_{n+1}\to0$ and $0<r_n<1$. Thus $\widetilde W_\omega\subset\Omega$. On the other hand, \eqref{eq:core} gives $\widetilde W_\omega\cap\Ucal=\varnothing$. {The disks for the three values of $\omega$ are pairwise disjoint, since $K_n$ has distinct limits on them. This completes the proof.}
\end{proof}

\subsection{{Construction of a sequence with $\Ucal\subsetneq\Omega$}}

We now realize the hypotheses of \Cref{thm:stable-disks} in the original H\'enon family.

\begin{theorem}\label{thm:proper-core-example}
For every fixed $0<c<1$, there is a positive sequence $(q_n)$ such that
$$
q_{n+1}/q_n\to2,\qquad {\rho_n\to0},\qquad
\sum_n\lambda_n<\infty,\qquad {\alpha_n\to0},
$$
and $\Ucal\subsetneq\Omega$. Moreover, $\Omega\setminus \Ucal$ contains three pairwise disjoint holomorphic disks, and $\Omega$ is parabolic. In particular, $\Omega$ is not a Short~$\mathbb C^2$.
\end{theorem}

\begin{proof}
Set $N_1=1$ and $\Lambda_1=c^{-2}$. Recursively define
$$
\ell_m:=\left\lceil\frac{1}{m^2\Lambda_m}\right\rceil,
\qquad
N_{m+1}:=N_m+\ell_m,
\qquad
\Lambda_{m+1}:=\Lambda_m c^{\,2^{N_{m+1}}/[m(m+1)]},
$$

and put

$$
q_n:=\frac{2^n}{m},
\qquad \text{for } N_m\le n<N_{m+1}.
$$

Since $\ell_m\ge1$, we have $N_m\ge m$ for every $m$. Recall that
$$
\lambda_0=c^{-q_1}=c^{-2}=\Lambda_1,
\qquad
\frac{\lambda_n}{\lambda_{n-1}}=c^{\,2q_n-q_{n+1}}.
$$

Inside the $m$-th $q$-block one has $q_{n+1}=2q_n$, whereas at the transition between the $m$-th and $(m+1)$-st blocks,

$$
2q_{N_{m+1}-1}-q_{N_{m+1}}
=
\frac{2^{N_{m+1}}}{m(m+1)}.
$$

It follows inductively that

$$
\lambda_n=\Lambda_m,
\qquad N_m-1\le n\le N_{m+1}-2.
$$

Write $\mu:=\log(1/c)>0$. For $m\ge2$ we have

$$
\log\frac1{\Lambda_m}
=
\mu\left(
-2+\sum_{j=1}^{m-1}\frac{2^{N_{j+1}}}{j(j+1)}
\right).
$$

Keeping only the last term and using $N_m\ge m$, we obtain

$$
\log\frac1{\Lambda_m}
\ge
\mu\left(\frac{2^{N_m}}{m(m-1)}-2\right)
\ge
\mu\left(\frac{2^m}{m(m-1)}-2\right).
$$

Hence

$$
\Lambda_m
\le
e^{2\mu}\exp\left(-\mu\frac{2^m}{m(m-1)}\right),
$$

so in particular

$$
\sum_{m=1}^{\infty}m\Lambda_m<\infty.
$$

By the definition of $\ell_m$,

$$
\frac1{m^2}
\le
\ell_m\Lambda_m
<
\frac1{m^2}+\Lambda_m.
$$

Therefore

$$
\sum_{n=0}^{\infty}\lambda_n
=
\sum_{m=1}^{\infty}\ell_m\Lambda_m
<\infty.
$$

Set

$$
S_{m+1}:=L_{N_{m+1}-1}=\sum_{j>m}\ell_j\Lambda_j.
$$

The preceding estimate gives

$$
\sum_{j>m}\frac1{j^2}
\le
S_{m+1}
<
\sum_{j>m}\frac1{j^2}+\sum_{j>m}\Lambda_j.
$$

Moreover,

$$
\sum_{j>m}\Lambda_j
\le
\frac1{m+1}\sum_{j>m}j\Lambda_j
=
o\left(\frac1m\right),
$$

while

$$
\sum_{j>m}\frac1{j^2}\asymp\frac1m.
$$

Thus

$$
S_{m+1}\asymp\frac1m.
$$

Suppose that $n$ belongs to the $m$-th $\lambda$-block, and let $k$ be the number of terms of that block remaining from $n$ onward. Then $1\le k\le\ell_m$ and $L_n=k\Lambda_m+S_{m+1}$.
Since

$$
k\Lambda_m
\le
\ell_m\Lambda_m
<
\frac1{m^2}+\Lambda_m,
$$
we obtain, uniformly on the $m$-th block, $L_n\asymp\frac1m$. Consequently,

$$
\log\frac1{L_n}=\log m+O(1).
$$

Since $\lambda_n=\Lambda_m$, the previous lower bound for $\log(1/\Lambda_m)$ yields

$$
0\le
\alpha_n
=
\frac{\log(1/L_n)}{\log(1/\lambda_n)}
\le
\frac{\log m+O(1)}
{\mu\bigl(2^m/[m(m-1)]-2\bigr)}
\longrightarrow0.
$$
Hence {$\alpha_n\to0$, and therefore $\alpha_*=0$.} 
\medskip

We next verify the asymptotic properties of $(q_n)$. On the $m$-th block, $\rho_n={q_n}/{2^n}=\frac1m$, so $\rho_n\to0$. Moreover, ${q_{n+1}}/{q_n}=2$
inside each block, while at a transition,

$$
\frac{q_{N_{m+1}}}{q_{N_{m+1}-1}}
=
\frac{2m}{m+1}
\longrightarrow2.
$$
Hence ${q_{n+1}}/{q_n}\longrightarrow2$. 

If $n$ belongs to the $m$-th block, then $m\le n$, and therefore

$$
q_n^{1/n}
=
2m^{-1/n}
=
2\exp\left(-\frac{\log m}{n}\right)
\longrightarrow2,
$$

since

$$
0\le\frac{\log m}{n}\le\frac{\log n}{n}\longrightarrow0.
$$

In particular,

$$
q_n\ge\frac{2^n}{n}\longrightarrow\infty.
$$

It remains to verify the hypotheses of \cref{thm:stable-disks}. On the $m$-th $\lambda$-block,

$$
r_n
=
\frac{\Lambda_m}{S_{m+1}+k\Lambda_m}.
$$
where {$k$ is again the number of terms of that block remaining from $n$ onward.}
Since $x\mapsto(S_{m+1}+\Lambda_m x)^{-2}$ is decreasing,

$$
\begin{aligned}
\sum_{\substack{n\text{ in}\\\text{block }m}}r_n^2
&=
\Lambda_m^2
\sum_{k=1}^{\ell_m}
\frac1{(S_{m+1}+k\Lambda_m)^2}\\
&\le
\Lambda_m^2
\int_0^\infty
\frac{dx}{(S_{m+1}+\Lambda_m x)^2}=
\frac{\Lambda_m}{S_{m+1}}.
\end{aligned}
$$

Since $S_{m+1}\asymp m^{-1}$,

$$
\sum_{m=1}^{\infty}\frac{\Lambda_m}{S_{m+1}}
\le
C\sum_{m=1}^{\infty}m\Lambda_m
<\infty,
$$

and consequently

$$
\sum_{n=0}^{\infty}r_n^2<\infty.
$$

Furthermore, $r_n$ is increasing within each $\lambda$-block. If

$$
{r_m^{\max}}:=\max\{r_n:N_m-1\le n\le N_{m+1}-2\},
$$

then

$$
{r_m^{\max}}\le\frac{\Lambda_m}{S_{m+1}}.
$$
For each $m$ we {then} have

\begin{align*}
\sum_{n=N_m-1}^{N_{m+1}-2}|r_{n+1}-r_n|
&=(r_{N_{m+1}-2}-r_{N_m-1})
 +|r_{N_{m+1}-1}-r_{N_{m+1}-2}|\\
&\le 2r_m^{\max}+r_{m+1}^{\max}.
\end{align*}

Hence

$$
\sum_{n=0}^{\infty}|r_{n+1}-r_n|
\le
3\sum_{m=1}^{\infty}{r_m^{\max}}
\le
3\sum_{m=1}^{\infty}\frac{\Lambda_m}{S_{m+1}}
<\infty.
$$

In particular,

$$
\sum_{m=0}^{\infty}|r_{2m+1}-r_{2m}|<\infty.
$$

Together with $\sum_n\lambda_n<\infty$ and $\sum_n r_n^2<\infty$, these are the hypotheses of \cref{thm:stable-disks} which together with \cref{cor:regular-short} completes the proof.
\end{proof}

%%%%%%%%%%%%%

\begin{remark}\label{rem:proper-core-not-classification}
{The strict inclusion $\Ucal\subsetneq\Omega$ shows that derivative normalization does not capture every point of the non-autonomous basin in this example. Nevertheless it does not determine whether $\Omega$ is biholomorphic to $\C^2$; another construction of a biholomorphism might still be possible.}
\end{remark}

\subsection{Concluding remarks}
By \Cref{thm:short}, a sequence relevant to \Cref{q:original} must satisfy $\liminf\rho_n=0$. When $q_{n+1}/q_n\to2$, \Cref{cor:regular-short} then shows that its non-autonomous basin is parabolic manifold. The first remaining question is whether every such basin is biholomorphic to $\C^2$. In the summable case, t.i. the sum in \eqref{eq:sum} is finite, \Cref{thm:alpha-capture} gives equality $\Ucal=\Omega$ when $\alpha_*>0$, whereas \Cref{thm:proper-core-example} shows that equality can fail when $\alpha_*=0$. Determining when equality holds at this endpoint would sharpen the capture criterion.

If $\sum_n\lambda_n=\infty$, the quadratic coefficients of $\Phi_n$ do not converge, so derivative normalization does not give the map used in \Cref{thm:summable-core}. Another normalization may be needed to study the full basin. Critical sequences for which $q_n^{1/n}\to2$ but $q_{n+1}/q_n$ does not converge present a further question: the product hypothesis in \Cref{thm:second-classification} need not follow from the critical growth rate alone. Also note that all main results presented here can easily be generalized to complex coefficients $a_n$.


\begin{thebibliography}{13}
\bibitem{AAM} Marco Abate, Alberto Abbondandolo, and Pietro Majer,
\emph{Stable manifolds for holomorphic automorphisms}, J. Reine Angew. Math. \textbf{690} (2014), 217--247.
\bibitem{AS} Aydin Aytuna and Azimbay Sadullaev,
\emph{Parabolic {Stein} manifolds}, Math. Scand. \textbf{114} (2014), no.~1, 86--109.
\bibitem{BV} Sayani Bera and Kaushal Verma,
\emph{Uniform non-autonomous basins of attraction}, Invent. Math. \textbf{238} (2024), no.~3, 995--1040.
\bibitem{Fshort} John Erik Forn\ae ss,
\emph{Short $\C^k$}, Complex Analysis in Several Variables---Memorial Conference of Kiyoshi Oka's Centennial Birthday, Kyoto/Nara 2001, Adv. Stud. Pure Math., vol.~42, Math. Soc. Japan, Tokyo, 2004, pp.~95--108.
\bibitem{FP} John Erik Forn\ae ss and Ratna Pal,
\emph{Notes on the Short $\C^k$'s}, J. Geom. Anal. \textbf{32} (2022), no.~4, Paper No.~133, 22~pp.
\bibitem{FS} John Erik Forn\ae ss and Berit Stens\o nes,
\emph{Stable manifolds of holomorphic hyperbolic maps}, Internat. J. Math. \textbf{15} (2004), no.~8, 749--758.
\bibitem{Klimek} Maciej Klimek,
\emph{Pluripotential theory}, London Mathematical Society Monographs. New Series, vol.~6, The Clarendon Press, Oxford University Press, New York, 1991.
\bibitem{Peters} Han Peters,
\emph{Perturbed basins of attraction}, Math. Ann. \textbf{337} (2007), no.~1, 1--13.
\bibitem{PS} Han Peters and Iris Marjan Smit,
\emph{Adaptive trains for attracting sequences of holomorphic automorphisms}, Complex Anal. Synerg. \textbf{1} (2015), Article No.~3.
\bibitem{PW} Han Peters and Erlend Forn\ae ss Wold,
\emph{Non-autonomous basins of attraction and their boundaries}, J. Geom. Anal. \textbf{15} (2005), no.~1, 123--136.
\bibitem{RR} Jean-Pierre Rosay and Walter Rudin,
\emph{Holomorphic maps from $\C^n$ to $\C^n$}, Trans. Amer. Math. Soc. \textbf{310} (1988), no.~1, 47--86.
\bibitem{Wold} Erlend Forn\ae ss Wold,
\emph{Fatou--{Bieberbach} domains}, Internat. J. Math. \textbf{16} (2005), no.~10, 1119--1130.
\bibitem{Zahariuta} V.~P. Zahariuta,
\emph{Spaces of analytic functions and complex potential theory}, Linear Topological Spaces and Complex Analysis (Aydin Aytuna, ed.), vol.~1, METU--T\"UB\.{I}TAK, Ankara, 1994, pp.~74--146.
\end{thebibliography}
\end{document}